\documentclass[11pt]{article}
\usepackage{amsthm}
\usepackage{amsfonts}
\usepackage{amsmath}
\usepackage{amssymb}
\usepackage{graphicx}
\usepackage{epstopdf}
\usepackage{filecontents}
\usepackage[left=2.5cm,right=2.5cm,top=2.5cm,bottom=2.5cm]{geometry}

\newtheorem{proposition}{Proposition}

\newtheorem{theorem}{Theorem}[section]

\theoremstyle{definition}
\newtheorem{definition}{Definition}
\theoremstyle{remark}

    \makeatletter
    \let\@fnsymbol\@arabic
    \makeatother

\title{$L_1$ spline fits via sliding window process : continuous and discrete cases}%
\author{Laurent Gajny\footnote{Arts et M\'etiers ParisTech, LSIS - UMR CNRS 7296, 8 Boulevard Louis XIV, 59046 Lille Cedex, France}, Olivier Gibaru$^{1,}$\footnote{INRIA Lille-Nord-Europe, NON-A research team, 40, avenue Halley 59650 Villeneuve d'Ascq}, Eric Nyiri$^1$}%
\date{}%

\begin{document}
\maketitle

\begin{abstract}
  Best $L_1$ approximation  of  the  Heaviside  function  and best $\ell_1$ approximation of multiscale univariate datasets by cubic splines have a Gibbs phenomenon. Numerical experiments show that it can be reduced by using $L_1$ spline fits which are best  $L_1$ approximations in an appropriate spline space obtained by the union of  $L_1$ interpolation splines. We prove here the existence of $L_1$ spline fits which has never been done to the best of our knowledge. Their  major disadvantage is  that obtaining  them can  be time consuming.  Thus we propose a sliding window method on seven nodes which is as efficient as the global method both for functions and datasets with abrupt changes of magnitude but within a linear complexity on the number of spline nodes.\\
  
  \textbf{Keywords :} Best approximation, $L_1$ norm,  shape preservation,  polynomial  spline,  Heaviside function, sliding window
\end{abstract}

\section{Introduction}

Over the past fifteen years, $L_1$ minimization-based methods have shown very interesting features for the interpolation and approximation of continuous or discontinuous function and irregular geometric data. In \cite{Moskona1995}, Moskona \emph{et al.} have shown the Gibbs phenomenon existing for best $L_1$ trigonometric approximation of the Heaviside function is lower than the one observed using the $L_2$ norm. Saff and Tashev have done a similar work leading to the same conclusion using polygonal lines \cite{SaffTashev1999}.\\
Similarly to classical cubic interpolation splines which minimize the $L_2$ norm of the second derivative, Lavery has defined cubic Hermite interpolation splines which minimizes the $L_1$ norm of the second derivative \cite{Lavery2000}. He has noted that this strategy enabled to delete completely the Gibbs phenomenon observed for classical $L_2$ cubic interpolation spline when applied on the Heaviside function. It has later been shown formally by Auquiert \emph{et al.} \cite{Auquiert2007}.\\
Further work has then focused on an appropriate combination of the best $L_1$ approximation functional and the variational $L_1$ functional used for the interpolation problem. Lavery has firstly proposed a linear combination of the two functionals and called the resulting splines $L_1$ smoothing splines \cite{Lavery2000b}. They do not introduce oscillation on multiscale univariate datasets contrary to $L_2$ smoothing splines. However, the regularization parameter used in the linear combination of $L_1$ functionals cannot be easily fixed.\\
Lavery has then proposed another kind of $L_1$ splines named $L_1$ spline fits \cite{Lavery2004}. They are best $L_1$ approximations in an appropriate spline space obtained by the union of $L_1$ interpolation splines. Like $L_1$ smoothing splines, they do not introduce oscillations but they have the asset that they do not need any additional parameter. Existence of such splines was not shown in \cite{Lavery2004}. We prove in this paper that $L_1$ splines fits at a given set of nodes exist for every function in $L_1[a,b]$.\\
One must admit that the intrinsic non-linearity of $L_1$ norm problems imply that a closed form solution is not available in general. A global numerical solvation is then currently used \cite{Lavery2000,Lavery2000b,Lavery2004,Auquiert2007b,Dobrev2010}. Another strategy has been introduced in 2011 by Nyiri, Auquiert and Gibaru for the interpolation problem \cite{Nyiri2011}. The algorithm they design is based on a sliding-window process. It consists in finding a set of local solutions on a limited number of successive points - five for the interpolation problem. By keeping appropriate information - the derivative at the middle point of the window - one can easily construct a global interpolating function which has similar shape preserving properties than the global solution. Moreover, this process enables to have a linear complexity algorithm and it can be parallelized. This process has also been applied in recent articles for a problem of approximation of data with prescribed error using $L_1$ norm \cite{Gajny2013,Gajny2014}. \\
Recently, Wang {et al.} proposed a method to compute $L_1$ spline fits with a global algorithm but based on a five-point interpolation rule to fix derivatives at the spline nodes \cite{Wang2014}. Indeed, the first derivative at a given node is determined using only its four neighbours while the value of the spline is determined by a minimization process on the whole dataset. We propose in this article another approach following the work in \cite{Nyiri2011,Gajny2013,Gajny2014}. We investigate to define an appropriate sliding window process to compute locally-computed $L_1$ spline fits close to the global one.\\
In the first section, we recall some generalities about $L_1$ cubic Hermite interpolation splines. We show that the union of such splines over all possible Lagrange interpolation is a closed set. This helps in the second section to show the existence of $L_1$ spline fits previously introduced in the literature. We introduce in section 3 and 4 sliding-window algorithms to determine a locally-computed $L_1$ spline fits and we compare them to each other. Conclusions are drawn in a last section.

\section{The set of $L_1$ cubic Hermite interpolation splines}
Let $(x_i,y_i)$, $i=1,\dots,n$, where $x_1<x_2<\dots<x_n$, be $n$ data points belonging to the graph of a function $f$. Let $Her(\mathbf{x})$  the space of cubic Hermite splines with nodes $\mathbf{x}=\{x_1,x_2,\dots,x_n\}$. A $L_1$ cubic Hermite interpolation spline of this data is a cubic Hermite spline $\gamma^*\in Her(\mathbf{x})$ a solution of :
\begin{equation}
 \min_{\gamma \in Her(\mathbf{x})} \int_{x_1}^{x_n} \vert \gamma''(x) \vert \mathrm{d}x,
\end{equation}
under the Lagrange interpolation constraints :
\begin{equation}
  \gamma(x_i)=f_i, \ i=1,2,\dots,n.
\end{equation}
Lavery has shown that a solution of this problem always exists. By mean of numerical experiments, he has noted that the resulting splines were very efficient to preserve the shape of the Heaviside function (see Figure \ref{L1L2interp}). Auquiert later has shown that a $L_1$ cubic Hermite interpolation spline with six knots or more with at least three knots on each part of the Heaviside function preserve both linearities of the Heaviside function and then do not lead to a Gibbs phenomenon \cite{Auquiert2007}. This is the major asset of $L_1$ cubic Hermite interpolation splines.\\

\begin{figure}[!h]
\centering
  \begin{minipage}{0.45\linewidth}
    \includegraphics[width=\linewidth]{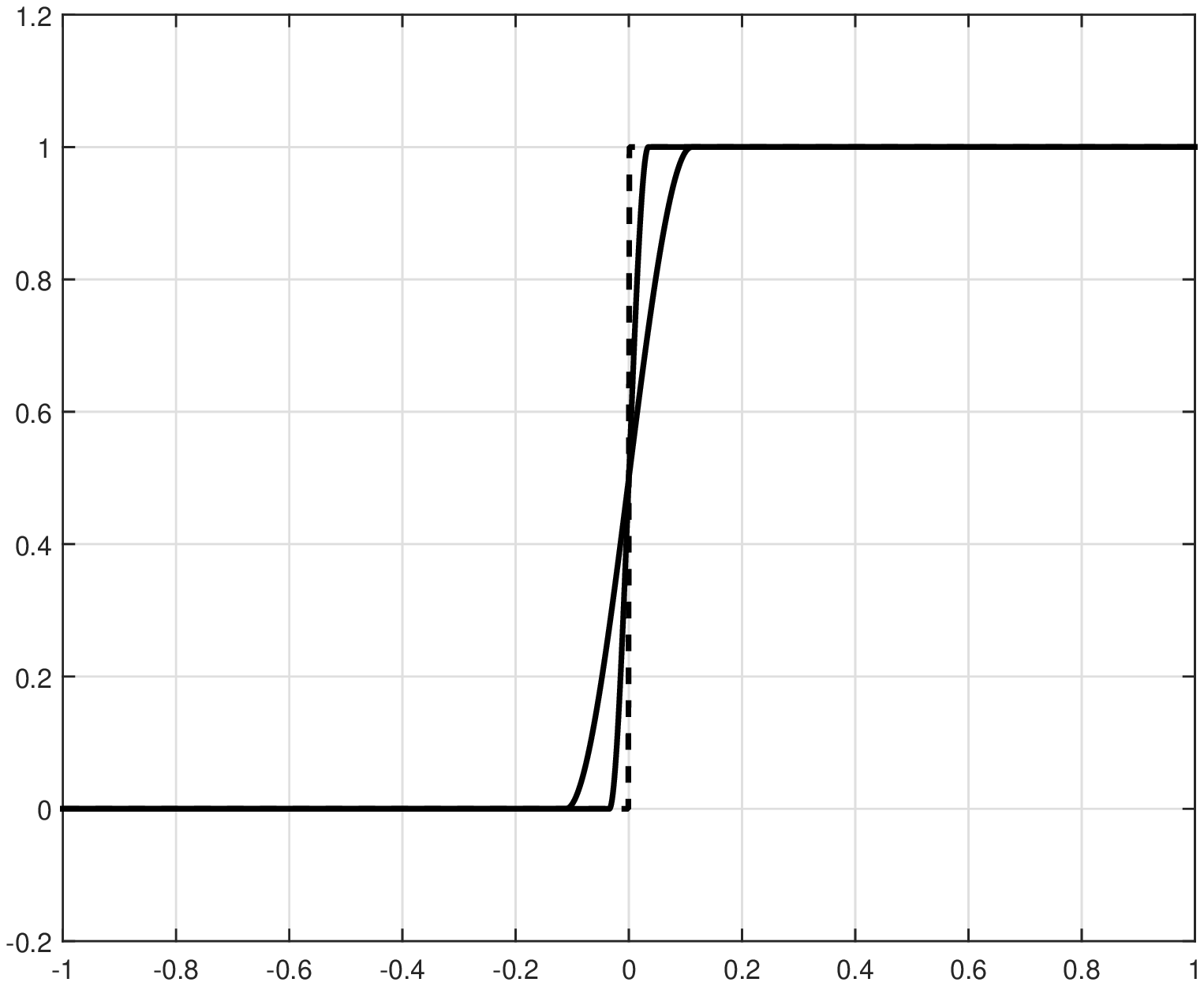}
  \end{minipage}
    \begin{minipage}{0.45\linewidth}
    \includegraphics[width=\linewidth]{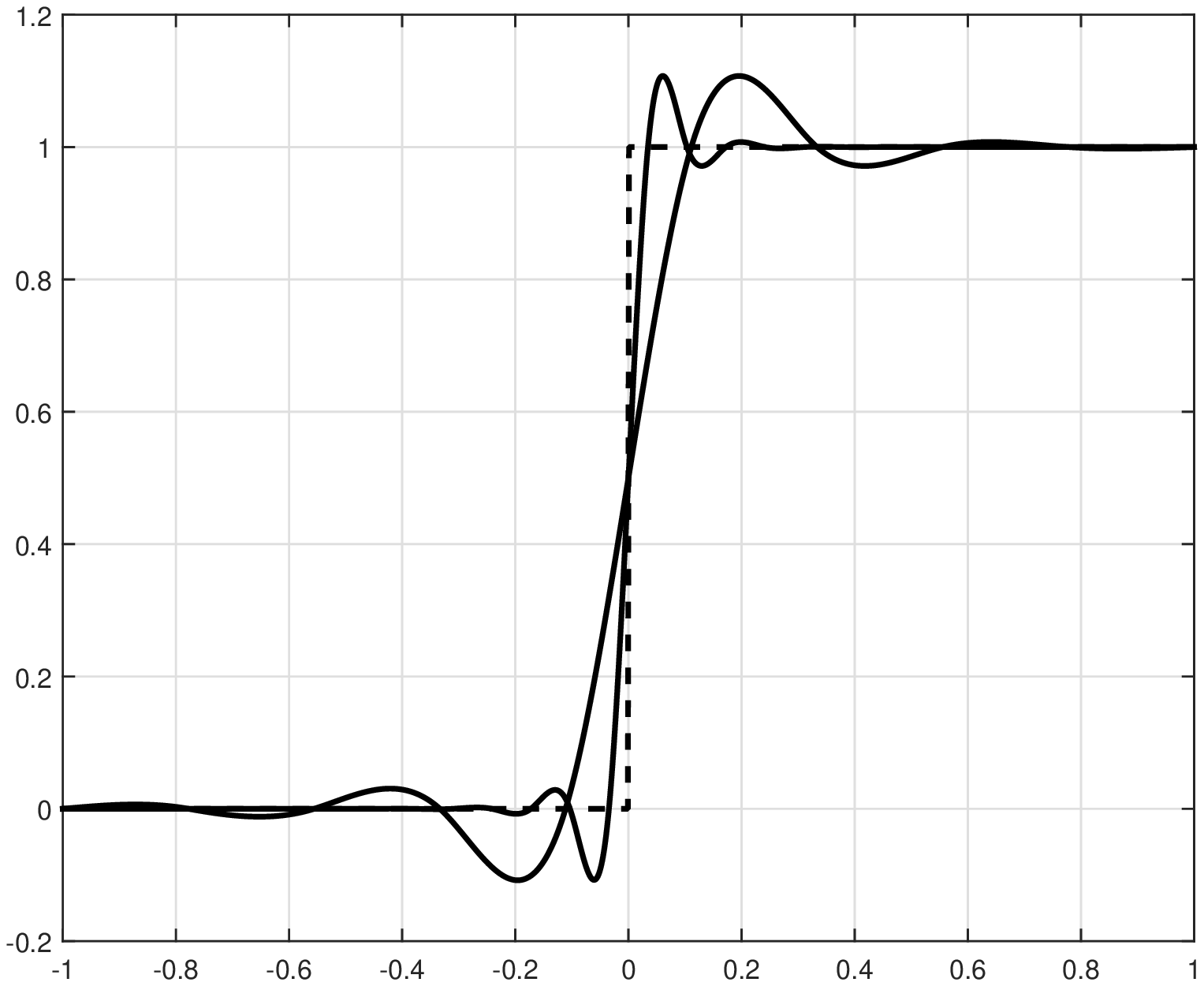}
  \end{minipage}
  \caption{$L_1$ (left) and $L_2$ (right) interpolation splines (solid lines) of the Heaviside function (dotted line) with 10 and 30 equally spaced knots.}
  \label{L1L2interp}
\end{figure}

We consider now the union of all $L_1$ cubic Hermite interpolation splines as follow :
\begin{equation}
  \mathcal{F}_\mathbf{x} = \bigcup_{\mathbf{y}\in \mathbf{R}^n}  \mathrm{argmin}\left\{ \int_{x_1}^{x_n} \vert \gamma''(x) \vert \ \mathrm{d}x, \ \gamma \in Her(\mathbf{x}), \ \gamma(x_i)=y_i, \ i=1,\dots,n\right\}.
\end{equation}
This set will be fundamental in the definition of $L_1$ spline fits. We give in the following an important property of this set which has never been proved before to the best of our knowledge and will be very important in the next section.

\begin{proposition}
  Given $\mathbf{x}=\{x_1<x_2<\dots<x_n\} \in \mathbf{R}^n $, the set  $\mathcal{F}_\mathbf{x}$ is closed.
\end{proposition}

\begin{proof}
Let us define a norm on  $Her(\mathbf{x})$. Let $s \in \overline{\mathcal{F}_{\mathbf{x}}}$ then by definition there exists a sequence :
\[\left(s_p\in \mathrm{argmin} \left\{\int_{a}^{b} |\gamma''(x)| \mathrm{d}x, \ s\in Her(\mathbf{x}),\ \gamma(x_k)=q_k^{(p)},\ k=1,\dots,m\right\}\right)_{p\in \mathbf{N}}\]
which converges to $s\in Her(\mathbf{x})$. For all $p \in \mathbf{N}$, $s_p$ is a cubic Hermite spline and is then defined by $2n$ coefficients $q_k^{(p)}$, $b_k^{(p)}$, $k=1,\dots,n$, respectively the values and the first derivative values of $s_p$ at abscissae $x_k$. By convergence hypothesis in $Her(\mathbf{x})$, there existe real values $q_k^*$, $b_k^*$, $k=1,\dots,n$ such that :
\begin{equation}
\begin{split}
  q_k^{(p)} & \underset{p\rightarrow +\infty}{\longrightarrow} q_k^*,\\
  b_k^{(p)} & \underset{p\rightarrow +\infty}{\longrightarrow} b_k^*.
\end{split}
\label{eq_conv}
\end{equation}
By the unicity of the limit, $s$ is defined by these $2n$ coefficients. We then show that the minimization property of the splines $s_p$ is stable by passing to the limit.\\
We deduce from \eqref{eq_conv} that $(s_p'')_{p\in \mathbf{N}}$ converges simply almost everywhere to $s''$. Moreover, for all $p\in\mathbf{N}$, $s_p''$ is piecewise linear. Then we can easily bound it on the interval $[a,b]$ by an integrable function. By the dominated convergence theorem, it follows that :
\begin{equation}
  \int_a^b |s_p''(x)| \ \mathrm{d}x \underset{p\rightarrow +\infty}{\longrightarrow} \int_a^b |s''(x)| \ \mathrm{d}x.
\end{equation}
Let $\gamma \in  Her({\mathbf{x}})$ such that $\gamma(x_k)=q_k^*, \ k=1,\dots,n$. By the first assertion in \eqref{eq_conv}, there exists a sequence $(\gamma_p \in Her({\mathbf{x}}))_{p\in\mathbf{N}}$ such that for all $p\in \mathbf{N}$ et $k=1,\dots,n$, $\gamma_p(x_k)=q_k^{(p)}$ that converges to $\gamma$. We easily show that :
\begin{equation}
  \int_a^b |\gamma_p''(x)| \ \mathrm{d}x \underset{p\rightarrow + \infty}{\longrightarrow} \int_a^b |\gamma''(x)| \ \mathrm{d}x.
\end{equation}
For all $p \in \mathbf{N}$, since $s_p\in \widetilde{\mathcal{S}}_{1,\mathbf{x},\mathbf{q}^{(n)}}$, it follows that :
\begin{equation}
  \int_a^b |s_n''(x)| \ \mathrm{d}x \le  \int_a^b |\gamma_p''(x)| \ \mathrm{d}x.
\end{equation}
By passing to the limite, we have that for all $\gamma \in  Her({\mathbf{x}})$ such that $\gamma(x_k)=q_k^*$ :
  \begin{equation}
  \int_a^b |s''(x)| \ \mathrm{d}x \le  \int_a^b |\gamma''(x)| \ \mathrm{d}x.
\end{equation}
We conclude that  $\mathcal{F}_{1,\mathbf{x}}$ is closed in $Her(\mathbf{x})$.
\end{proof}

\section{Best approximation using $L_1$ spline fits}

Let us first define these splines introduced in \cite{Lavery2004}.
\begin{definition}
Given a function $f\in L_1[a,b]$, $a,b \in \mathbf{R}$ and a set of knots $\mathbf{x}=\{a=x_1<x_2<\dots<x_n=b\}$, a $L_1$ spline fit of the function $f$ at knots $\mathbf{x}$ is a best $L_1$ approximation of $f$ in $\mathcal{F}_\mathbf{x}$. In other words, it is a solution of :
\begin{equation}
  \min_{s \in \mathcal{F}_\mathbf{x}} \int_{a}^{b} \vert s(x) - f(x)\vert \ \mathrm{d}x.
\end{equation}
\end{definition}
We prove with the next theorem that $L_1$ spline fits are well defined.
\begin{theorem}
  $L_1$ splines fit exist for every function $f\in L_1[a,b]$ and every set of knots $\mathbf{x}=\{a=x_1<x_2<\dots,x_n=b\}$.
\end{theorem}

\begin{proof}
 Let us give $f\in L_1[a,b]$ and a set of knots $\mathbf{x}=\{a=x_1<x_2<\dots<x_n=b\}$. Since $\mathcal{F}_{1,\mathbf{x}}$ is closed in the finite dimensional subspace $Her(\mathbf{x})$ of $L_1[a,b]$, there exists a best $L_1$ approximation of $f$ in $\mathcal{F}_{1,\mathbf{x}}$. \hfill $\square$
\end{proof}

One can easily define an equivalent tool using exclusively $L_2$ norm and called $L_2$ splines fit. We compare both methods in Figure \ref{fig_heaviside_glob}.

\begin{figure}[!h]
\centering
\begin{minipage}{0.45\linewidth}
\includegraphics[width=\linewidth]{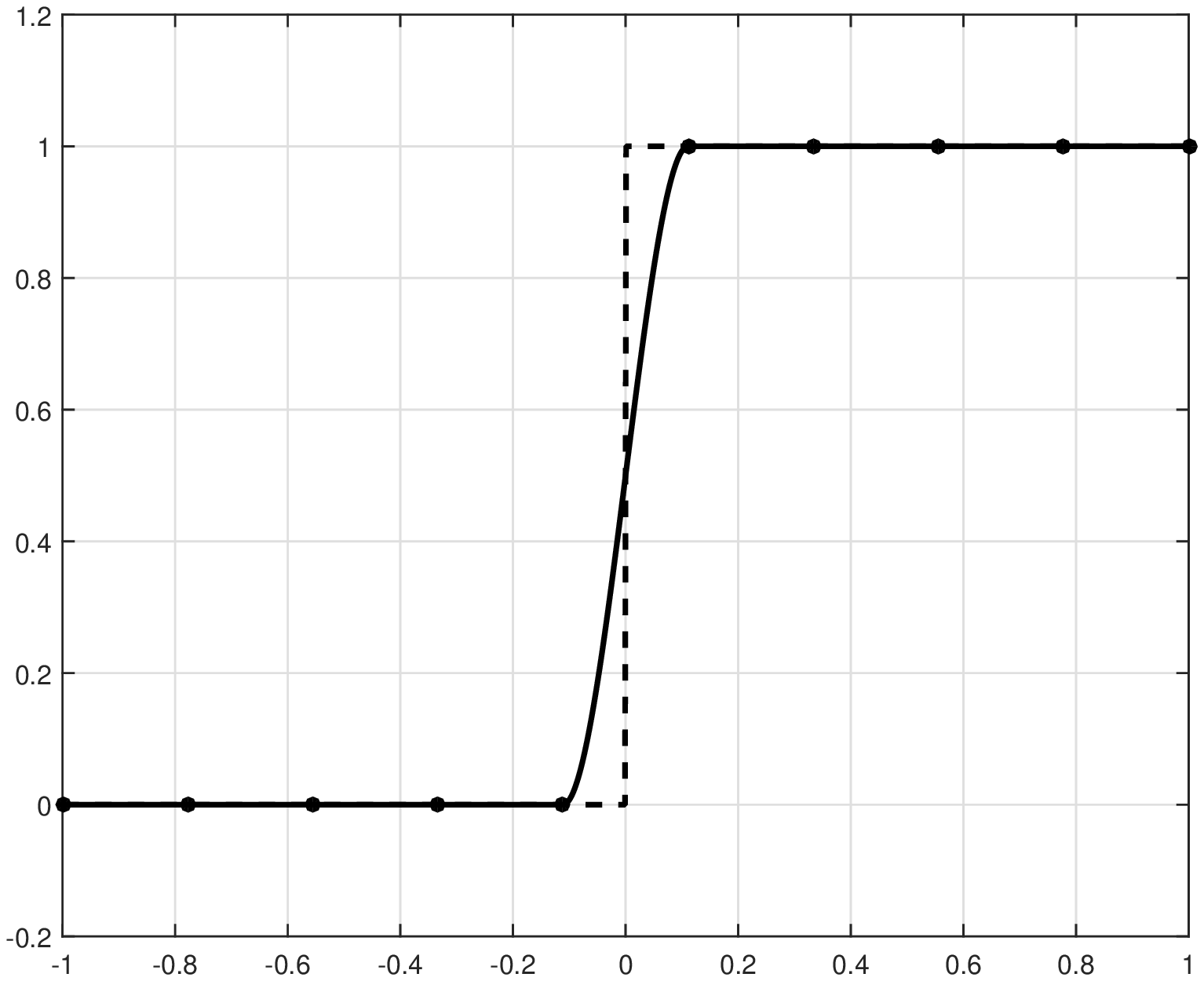}
\end{minipage}
\begin{minipage}{0.45\linewidth}
\includegraphics[width=\linewidth]{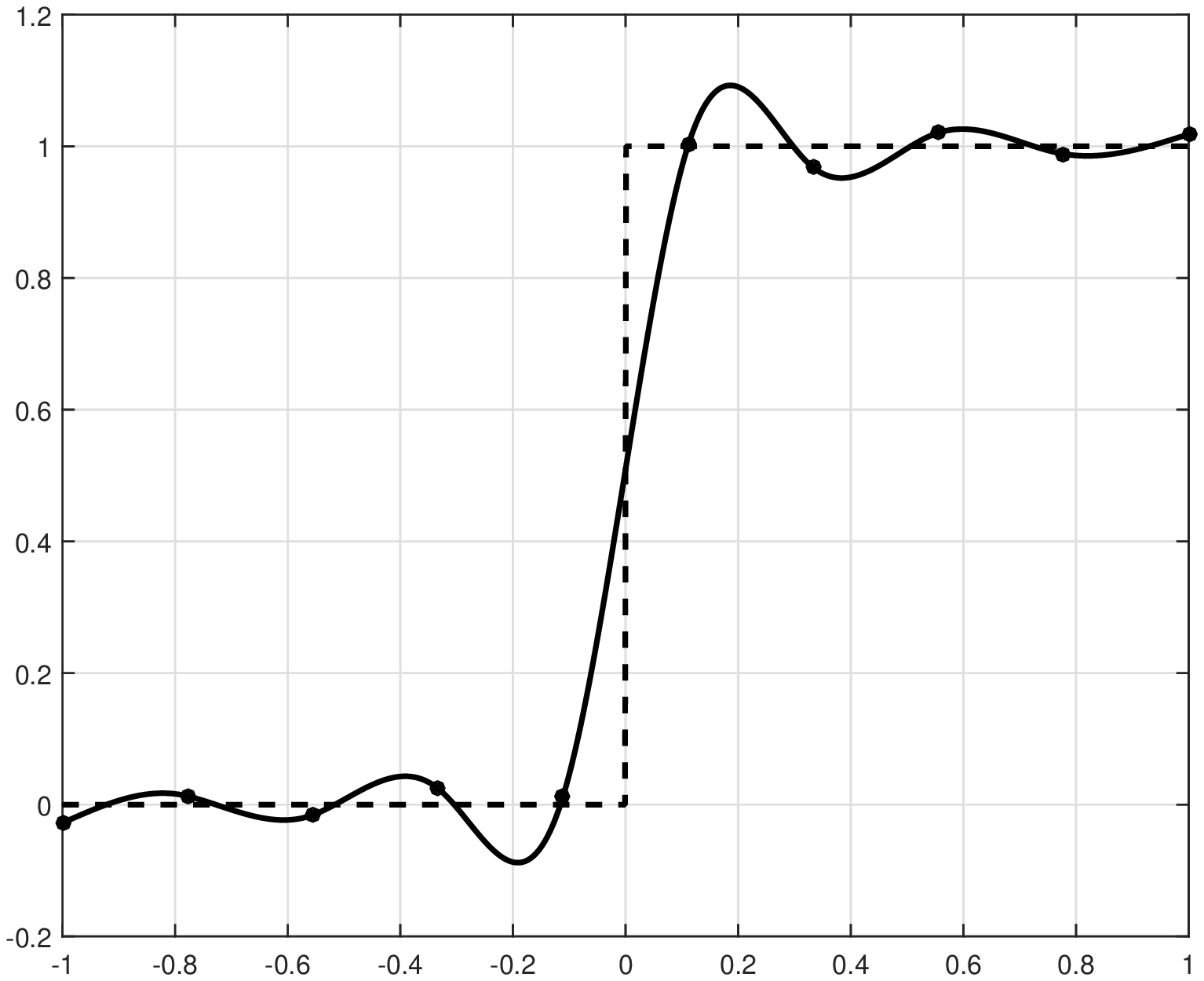}
\end{minipage}

\caption{Global $L_1$ spline fits (left) and global $L_2$ spline fits (right) of the Heaviside function with ten equally spaced knots.}
\label{fig_heaviside_glob}
\end{figure}

 $L_1$ spline fits has been previously defined for discrete data \cite{Lavery2004}. Let $(\hat{x}_i,\hat{y}_i), i=1,\dots,m$ be $m$ data points where $m\ge n$. A $L_1$ spline fits of this dataset is a best $\ell_1$ approximation of them in $\mathcal{F}_\mathbf{x}$. In other words, it is a solution of :
\begin{equation}
  \min_{s \in \mathcal{F}_\mathbf{x}} \sum_{i=1}^m \vert s(\hat{x}_i) - \hat{y}_i \vert.
\end{equation}
As in the continuous case, these splines exist since they are solutions of a best approximation problem in a closed set of a finite dimensional subpace of a normed linear space.
The results presented in Figure \ref{fig_L1SFG} indicate that $L_1$ spline fits preserve well the shape of multiscale data contrary to $L_2$ spline fits. Moreover, $L_1$ spline fits do not require human intervention to choose a parameter that balances weights of the approximation functional and the variational functional. However, the computational cost of $L_1$ spline fits is generally higher than the one of $L_1$ smoothing spline and more obviously of least squares methods.

\begin{figure}[!h]
\centering
\begin{minipage}{\linewidth}
\includegraphics[width=\linewidth]{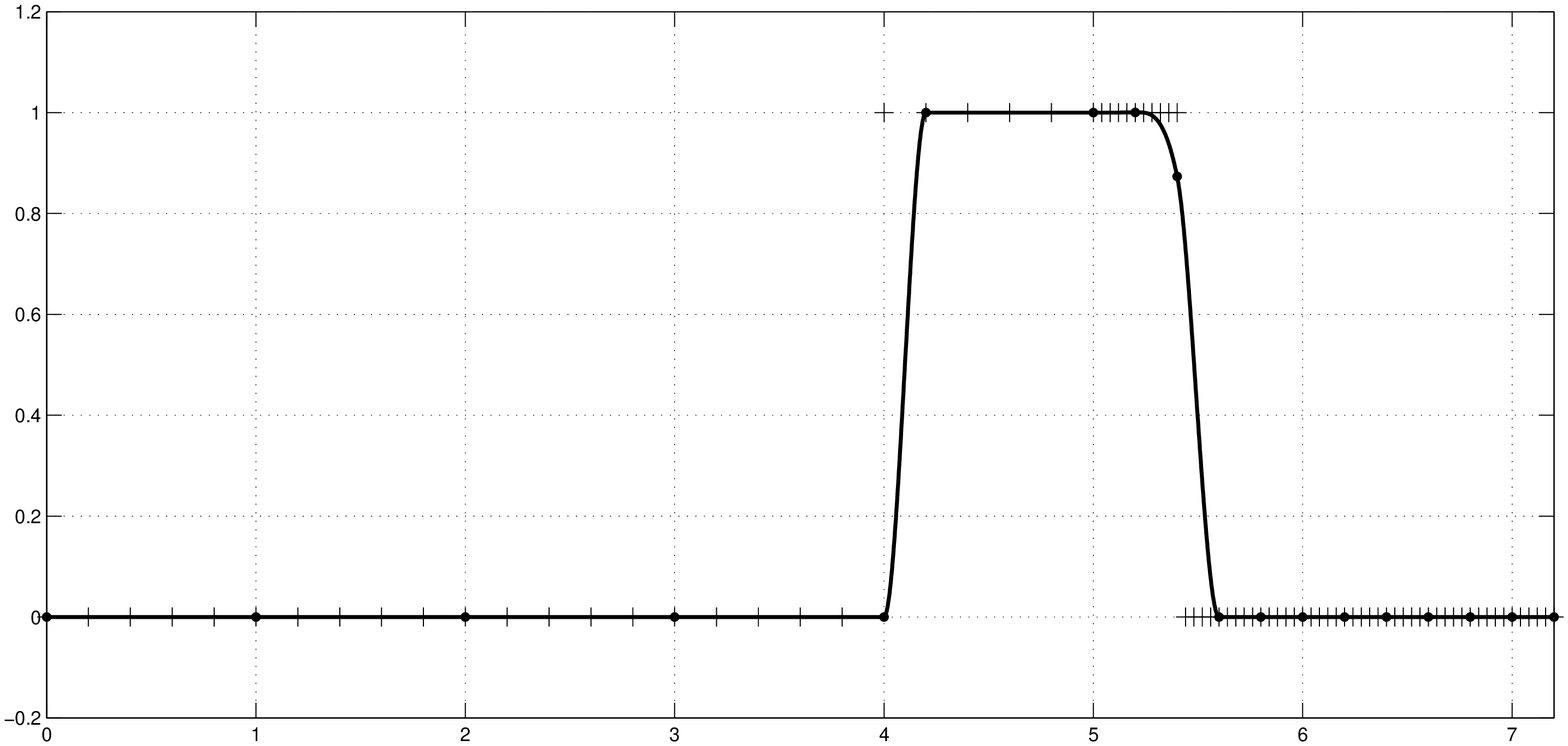}
\end{minipage}
\begin{minipage}{\linewidth}
\includegraphics[width=\linewidth]{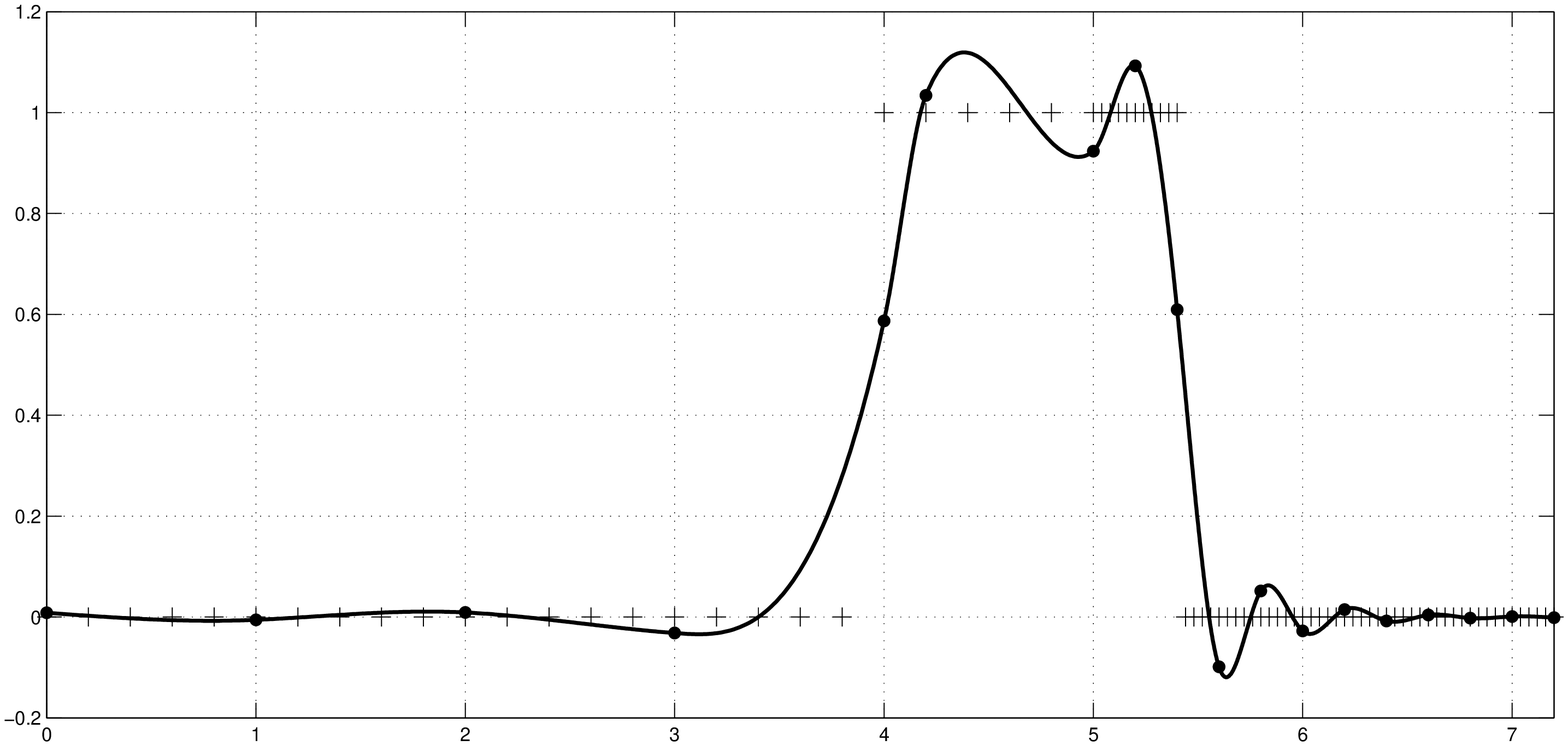}
\end{minipage}
\caption{Global $L_1$ spline fits (top) and global $L_2$ spline fits (bottom).}
\label{fig_L1SFG}
\end{figure}

\section{Sliding window algorithms for $L_1$ spline fits}

\subsection{Best approximation of functions}

We define sliding window methods with window size $m=3,\ 5,\ 7$ that we call respectively $L_1$SFL3, $L_1$SFL5 and $L_1$SFL7. For all set of $m$ consecutive knots $\mathbf{x}_{i,m}=\{x_{i-\lfloor \frac{m}{2}\rfloor},\dots,x_i,\dots,x_{i+\lfloor \frac{m}{2}\rfloor}\}$, we will determine numerically a cubic Hermite spline $s_{i,m}^*$ solution of :
\begin{equation}
  \min_{\gamma\in \mathcal{F}_{\mathbf{x}_{i,m}}} \int_{x_{i-\lfloor \frac{m}{2}\rfloor}}^{x_{i+\lfloor \frac{m}{2}\rfloor}} \vert \gamma(x) - f(x) \vert \ \mathrm{d}x.
\end{equation}
 Then we keep only middle information $z_{i}=s_{i,m}^*(x_i)$ and  $b_{i}=s_{i,m}^{*'}(x_i)$.
\begin{figure}[!h]
\centering
\begin{minipage}{0.45\linewidth}
\centering
  \includegraphics[width=\linewidth]{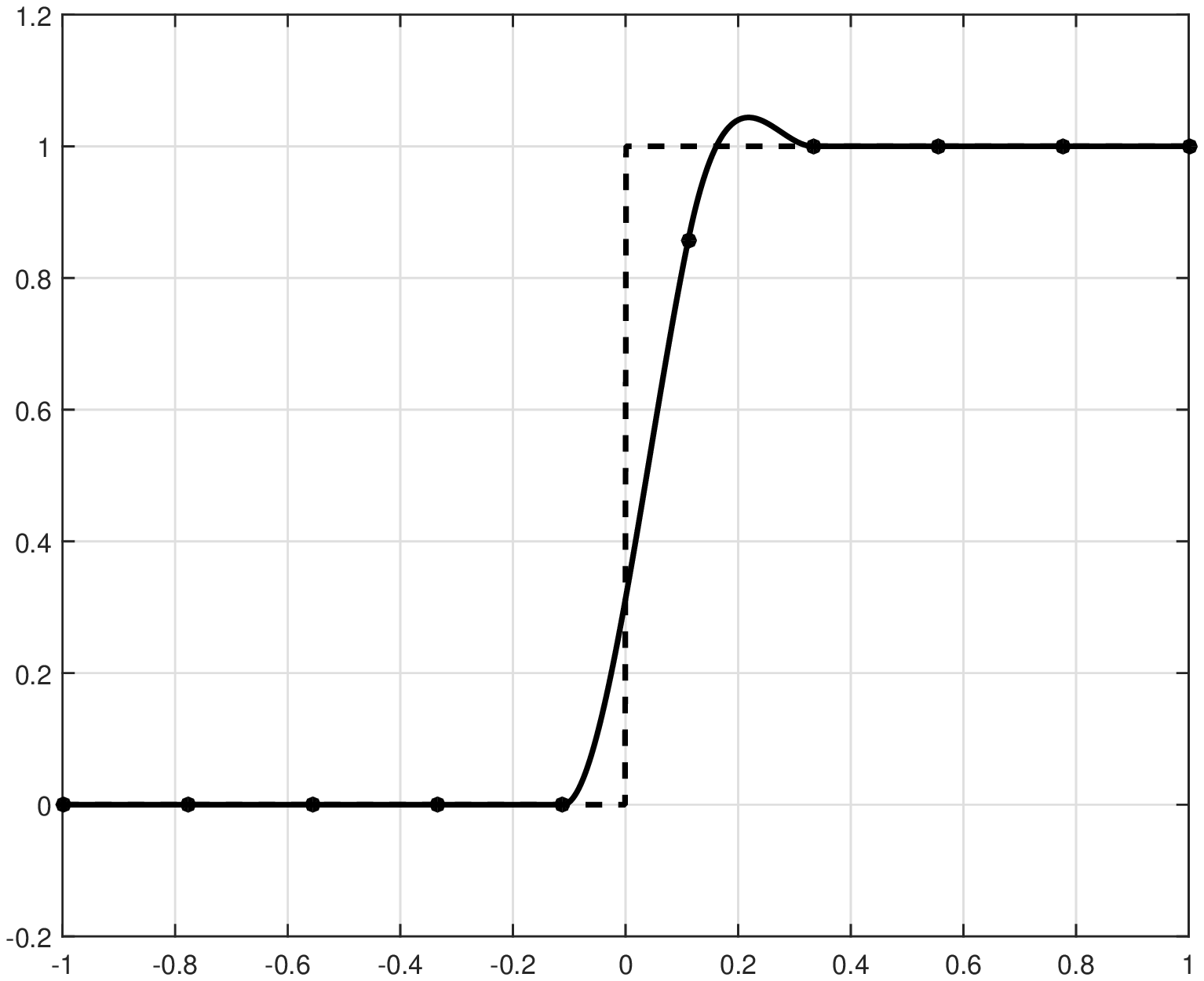}

Continuous $L_1$SFL3
\end{minipage}
\begin{minipage}{0.45\linewidth}
\centering

\includegraphics[width=\linewidth]{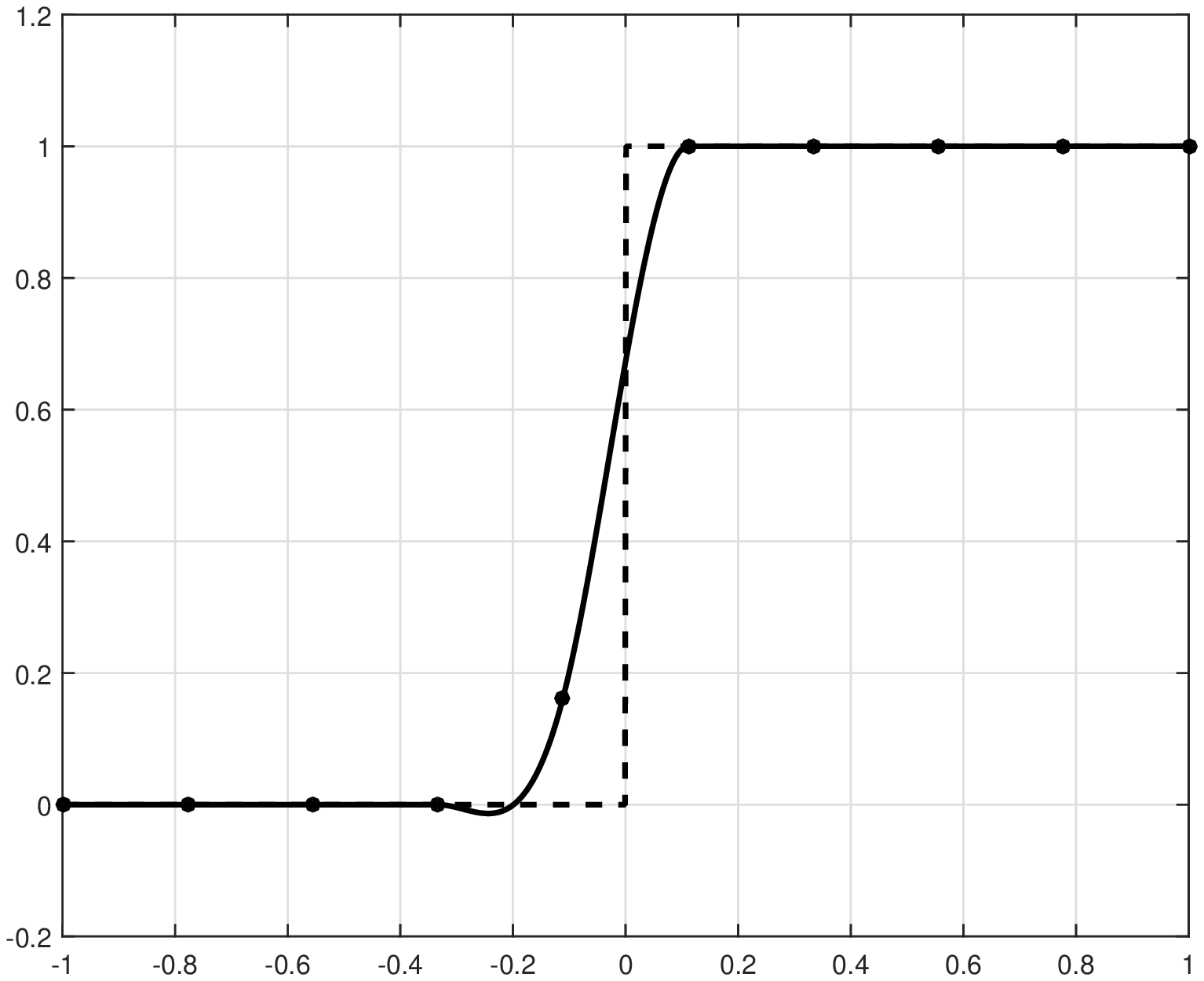}

Continuous $L_1$SFL5
\end{minipage}
\begin{minipage}{0.45\linewidth}
\centering
\includegraphics[width=\linewidth]{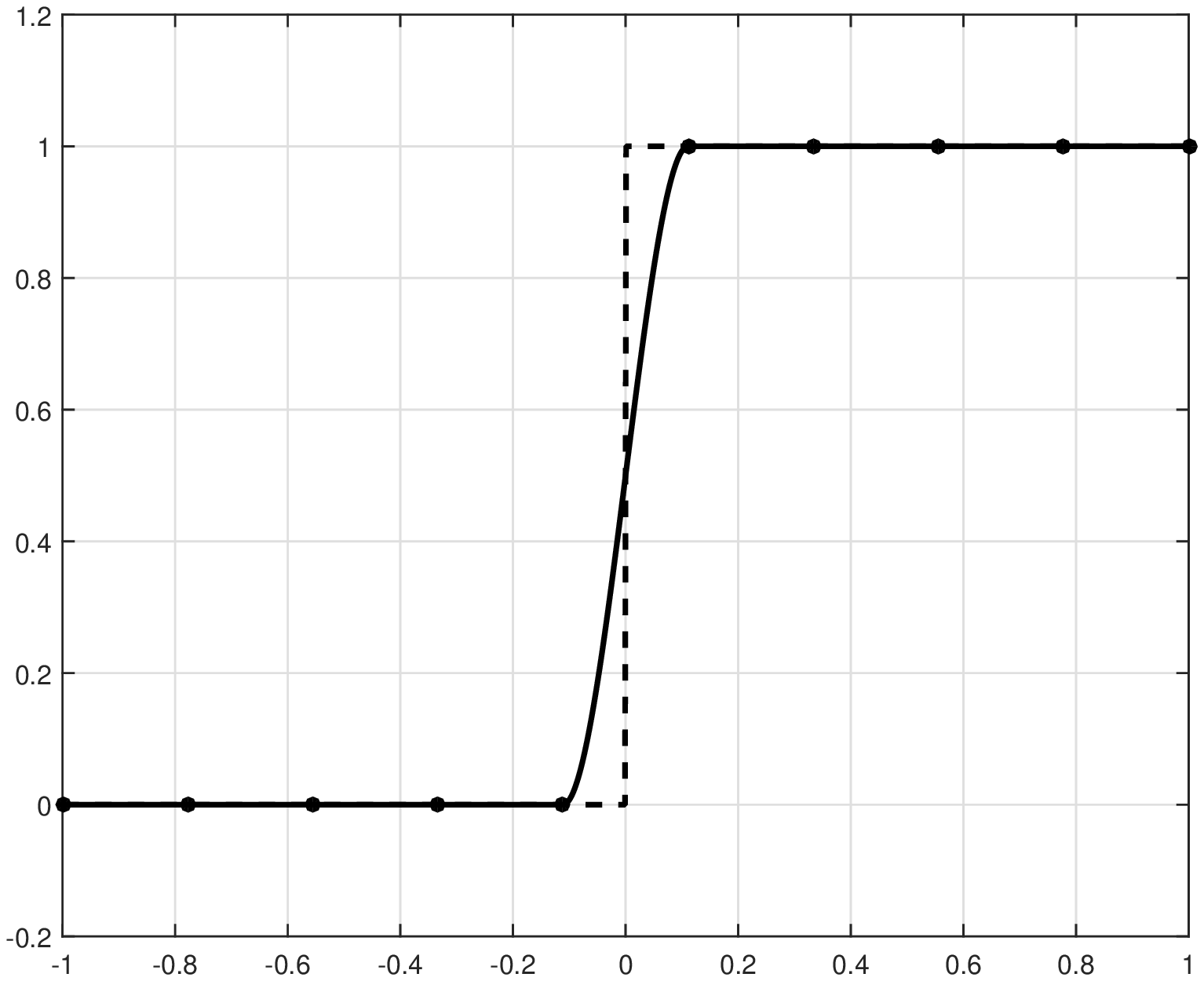}

Continuous $L_1$SFL7
\end{minipage}
\caption{$L_1$ spline fits computed by the three proposed sliding windows methods on the Heaviside function with ten equally spaced knots.}
\label{fig_Heaviside_cont}
\end{figure}

These methods have been tested on the Heaviside function with ten equally spaced knots and the results are summarized in Figure \ref{fig_Heaviside_cont}. The three-point and five-point methods fail to approximate linear shape on both side of the discontinuity. We are facing here typical cases of non-invariance of the numerical solution by rotation of the data. On both side of the discontinuities, the two windows considered are similar geometrically and should lead the same solution. Since on one side, we are able to preserve linearity by the three-point and the five-point methods, we should be able to do it on the other side. Further work will be done to make these methods invariant by rotation. The seven-point methods seems to more robust to rotation of data and so should be preferred. In this case, the seven-point solution and the global solution are identical.\\

We have also made some test about computing time and a comparison between the methods. The results are summarized in the graph in Figure \ref{fig_CPU_cont}. We can notice a great improvement of computing time when using local methods. The faster is of course the three-point method. We also notice a dual phenomenon in these results. The numerical solvation is different whether we have even or odd number of knots. It is linked to the fact of having a knot or not at the discontinuity.\\

\begin{figure}[!h]
\centering
\includegraphics[width=\linewidth]{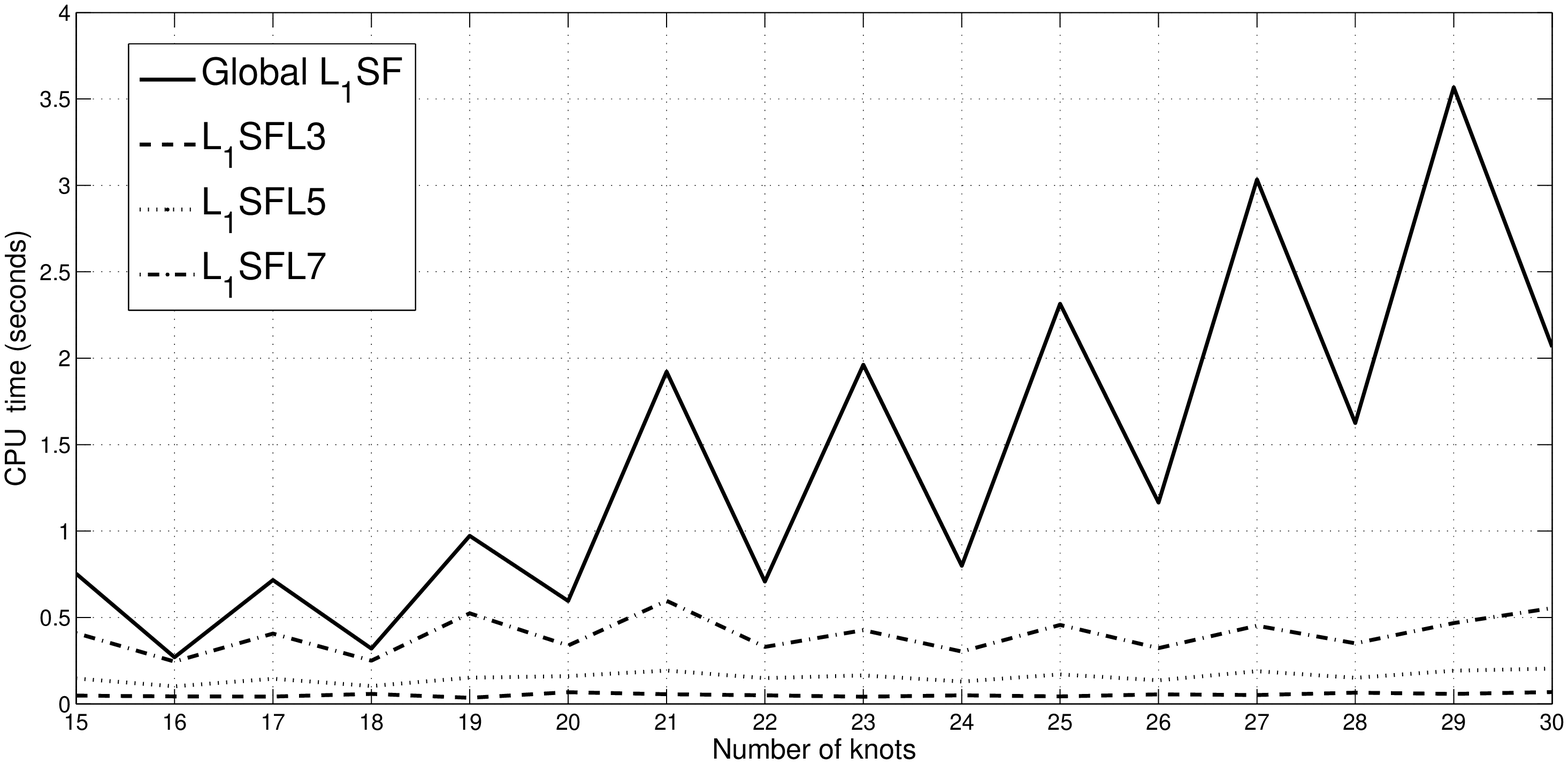}
\caption{Comparison of computational times between global and local methods}
\label{fig_CPU_cont}
\end{figure}
Regarding the graphical and computational time results, the seven-point method is a good compromise. We will confirm this tendency by the study of the discrete case.

\subsection{Best approximation of discrete data}

In this section, we apply the three-point, five-point and seven-point methods to discrete multiscale data. In other words, for all set of $m$ consecutive knots $\mathbf{x}_{i,m}=\{x_{i-\lfloor \frac{m}{2}\rfloor},\dots,x_i,\dots,x_{i+\lfloor \frac{m}{2}\rfloor}\}$, we will determine numerically a cubic Hermite spline $s_{i,m}^*$ solution of :
\begin{equation}
  \min_{\gamma\in \mathcal{F}_{\mathbf{x}_{i,m}}}  \sum_{j=i-\lfloor \frac{m}{2}\rfloor}^{i+\lfloor \frac{m}{2}\rfloor} \vert s(\hat{x}_j) - \hat{y}_j \vert.
\end{equation}
 Then we only keep information at the middle point of the window $z_{i}=s_{i,m}^*(x_i)$ and  $b_{i}=s_{i,m}^{*'}(x_i)$. The results are illustrated in Fig. \ref{fig_dataset1_L1SFL}, \ref{fig_dataset2_L1SFL} and \ref{fig_dataset3_L1SFL}. While the three point and the seven-point methods give smooth curves, the five-point method highly fails. In Fig. \ref{fig_dataset1_L1SFL}, we can notice an undershoot phenomenon and in Fig. \ref{fig_dataset2_L1SFL}, oscillations are created.\\

\begin{figure}[p]
\centering
\includegraphics[width=0.9\linewidth]{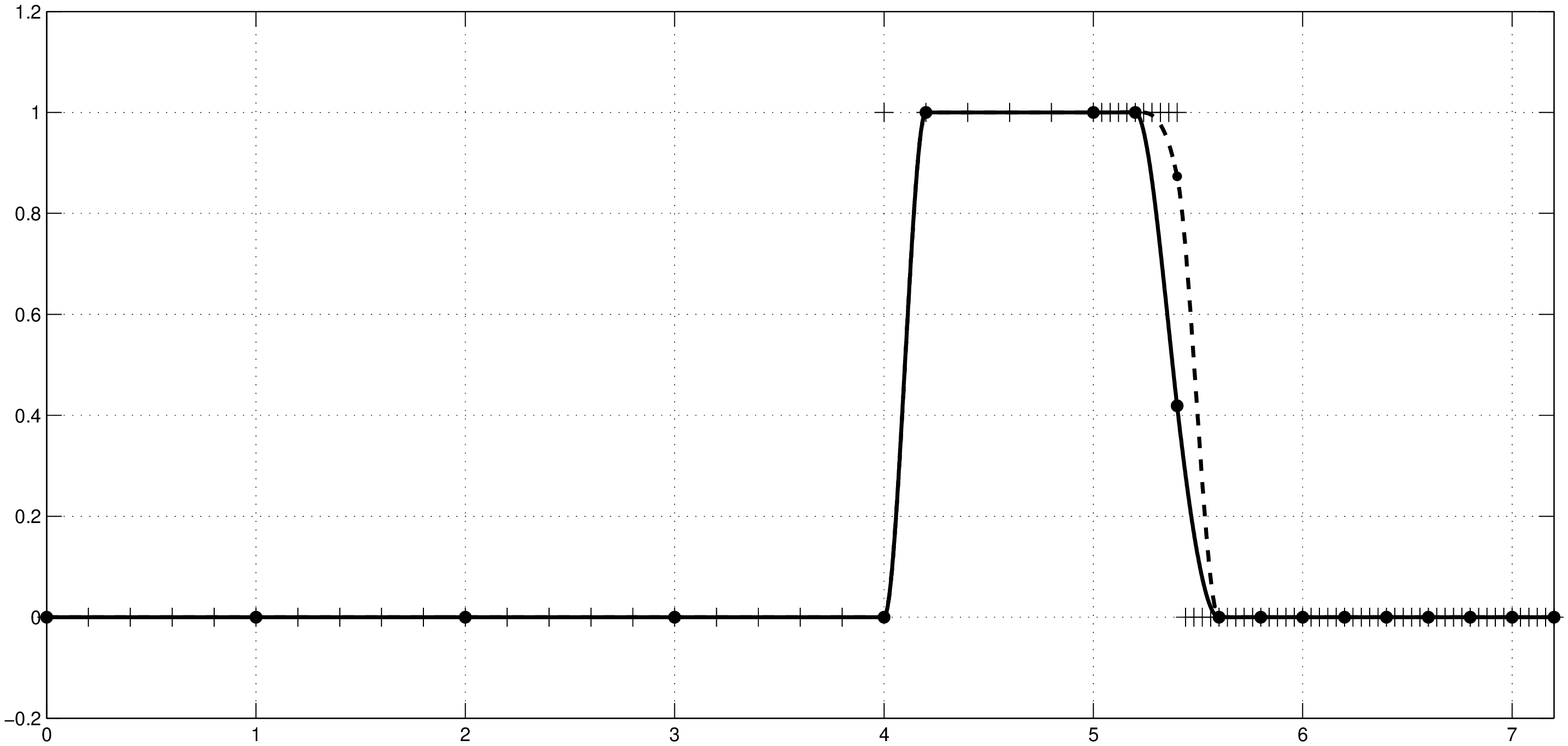}

Discrete $L_1$SFL3\\

\includegraphics[width=0.9\linewidth]{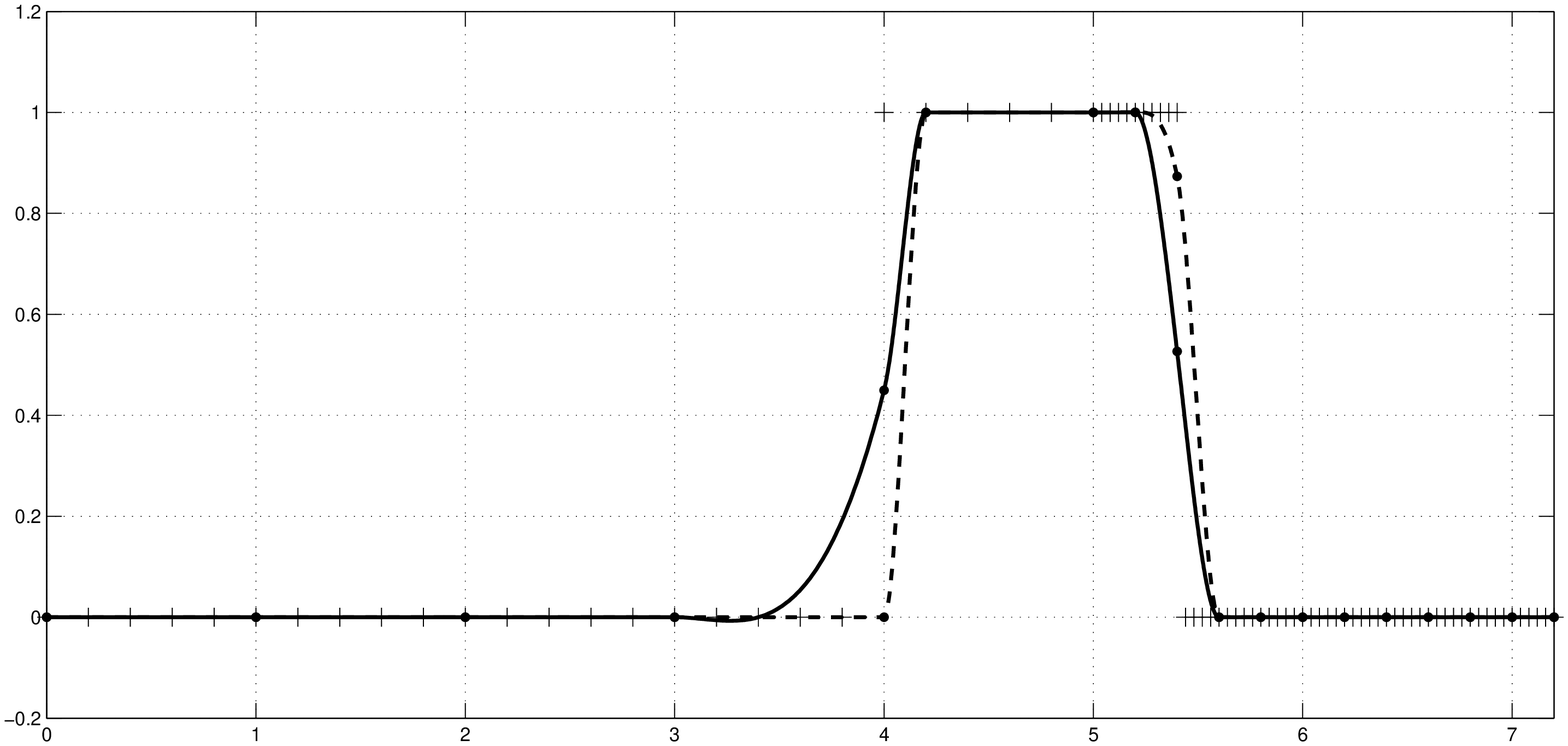}

Discrete $L_1$SFL5\\

\includegraphics[width=0.9\linewidth]{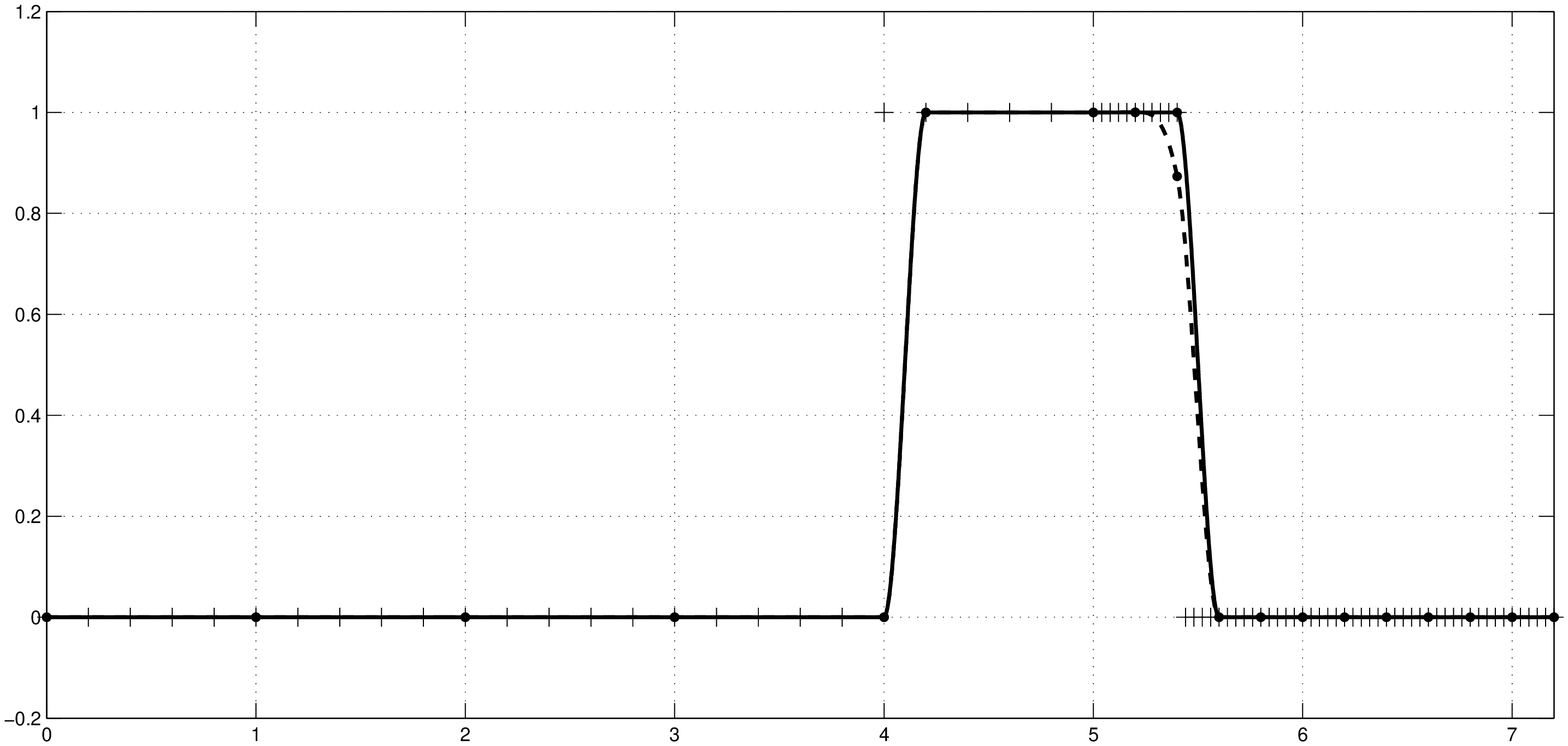}

Discrete $L_1$SFL7
\caption{Local (solid lines) and global (dashed line) $L_1$ spline fits on a multi-scale data set.}
\label{fig_dataset1_L1SFL}
\end{figure}

\begin{figure}[p]
\centering
\includegraphics[width=0.9\linewidth]{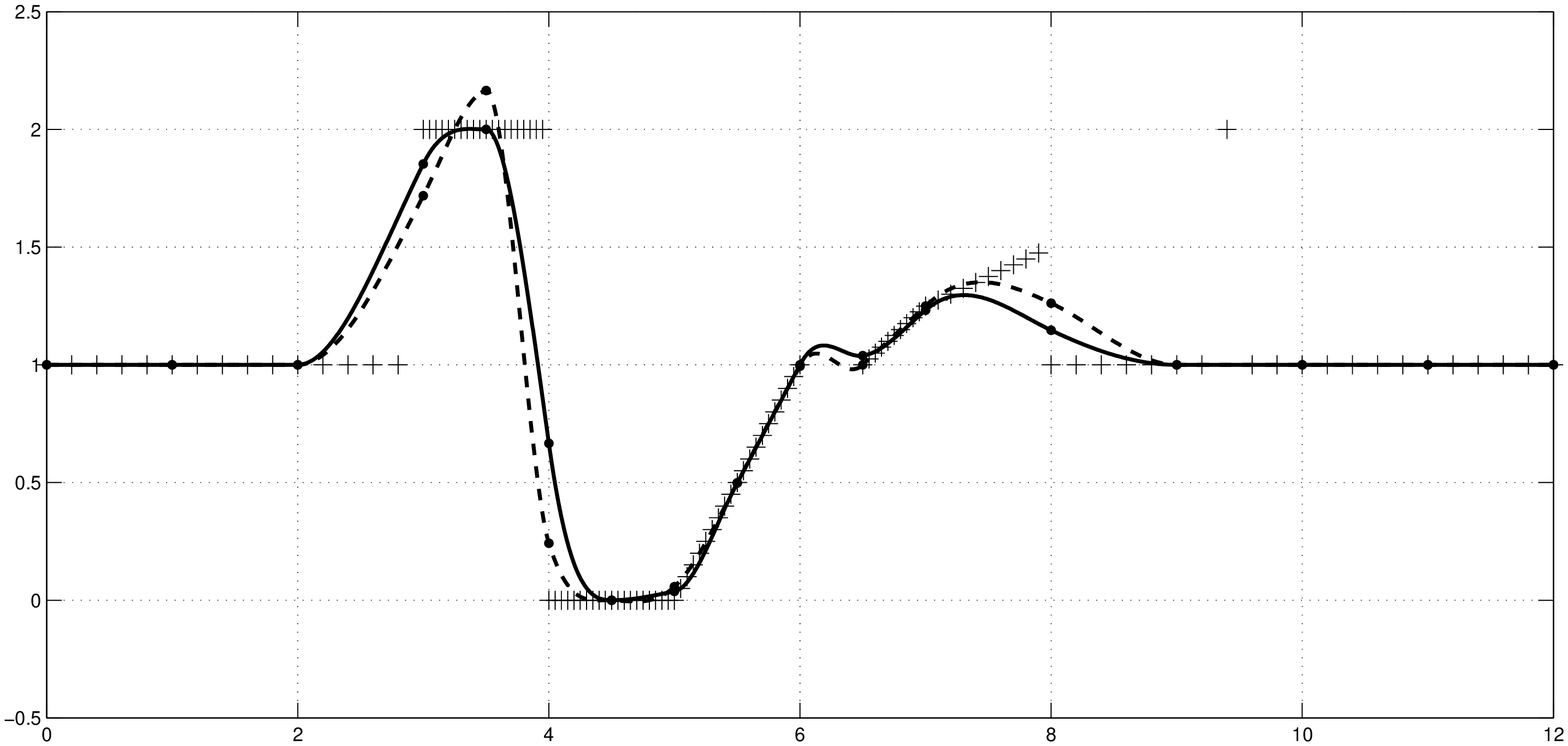}

Discrete $L_1$SFL3\\

\includegraphics[width=0.9\linewidth]{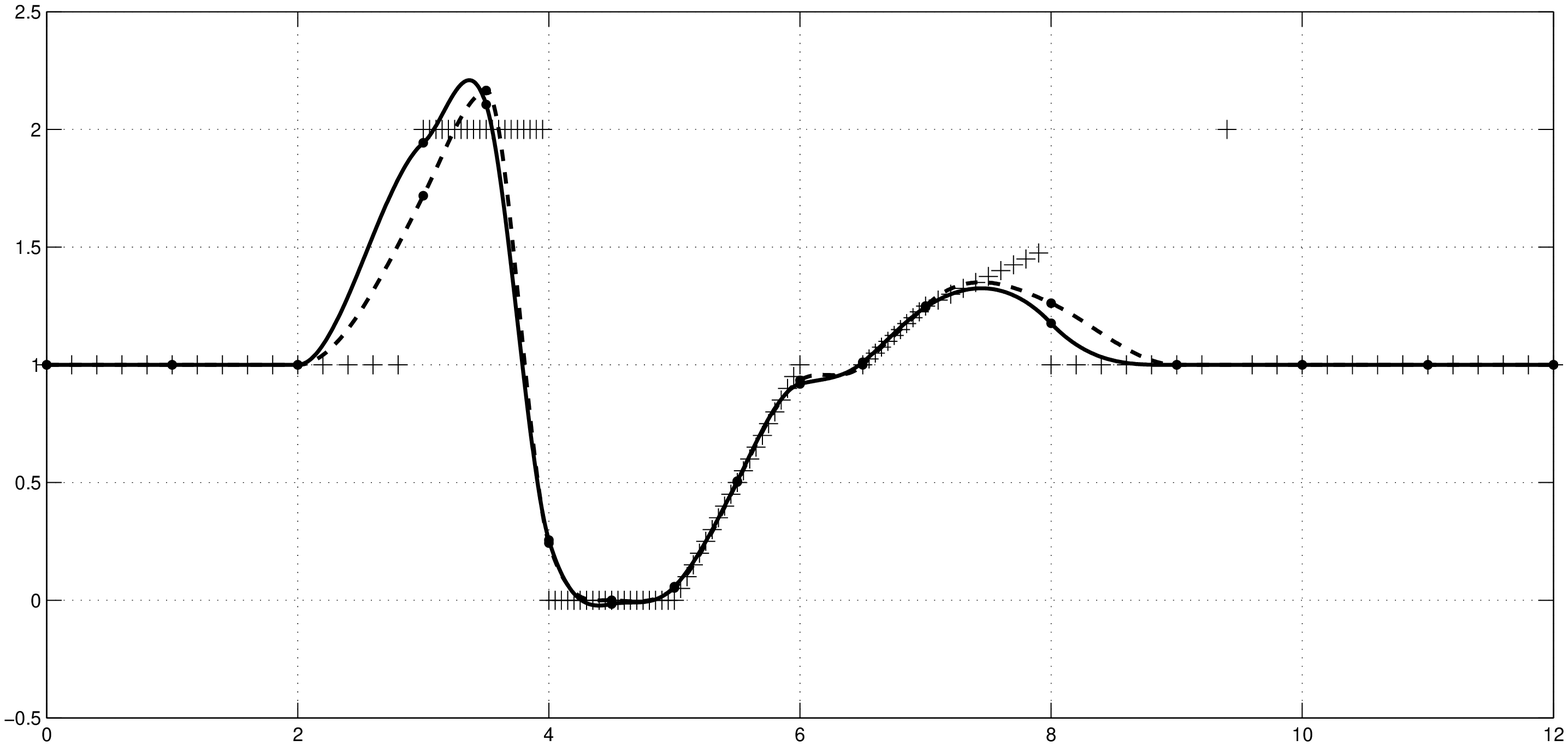}

Discrete $L_1$SFL5\\

\includegraphics[width=0.9\linewidth]{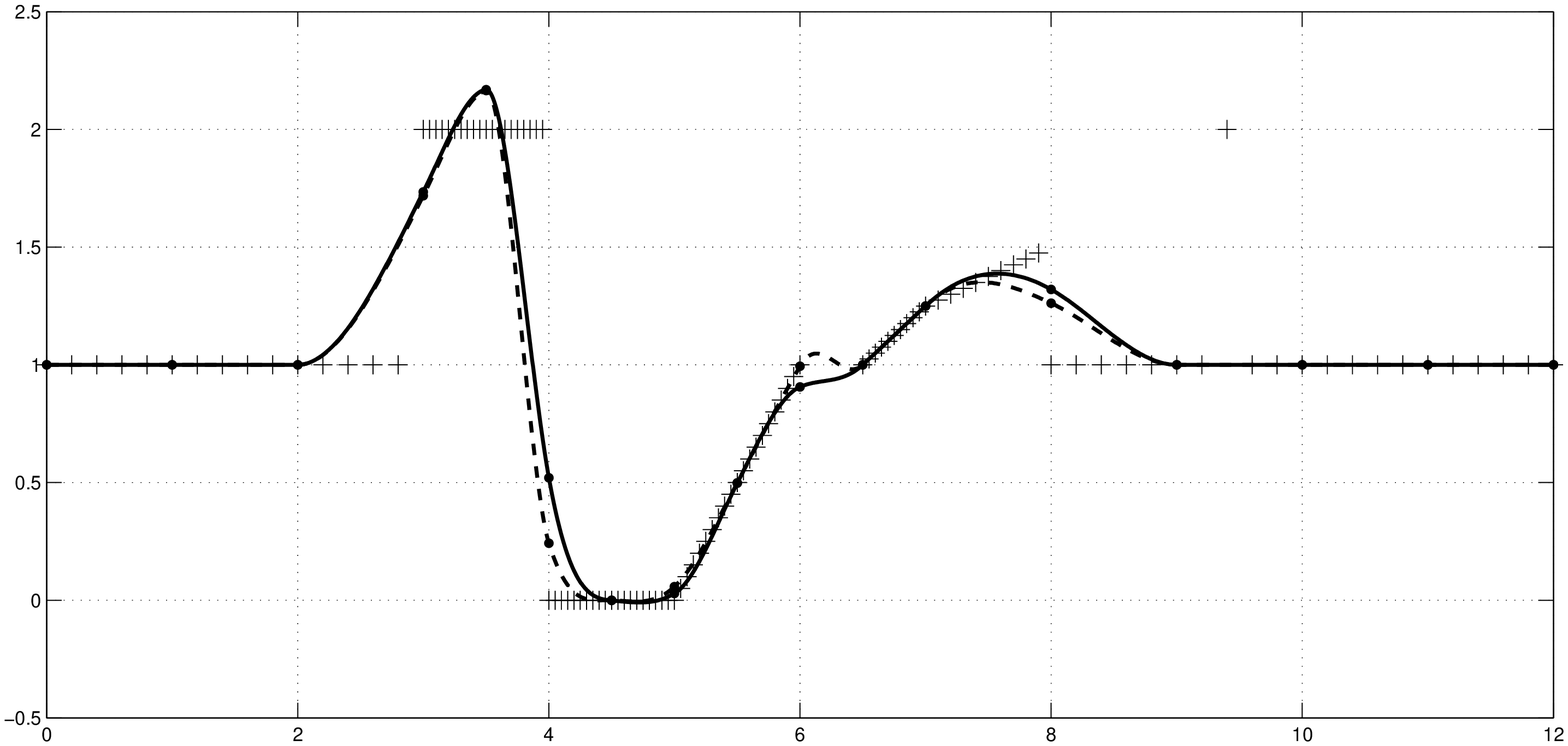}

Discrete $L_1$SFL7
\caption{Local (solid lines) and global (dashed line) $L_1$ spline fits on a multi-scale data set.}
\label{fig_dataset2_L1SFL}
\end{figure}

\begin{figure}[p]
\centering
\includegraphics[width=0.9\linewidth]{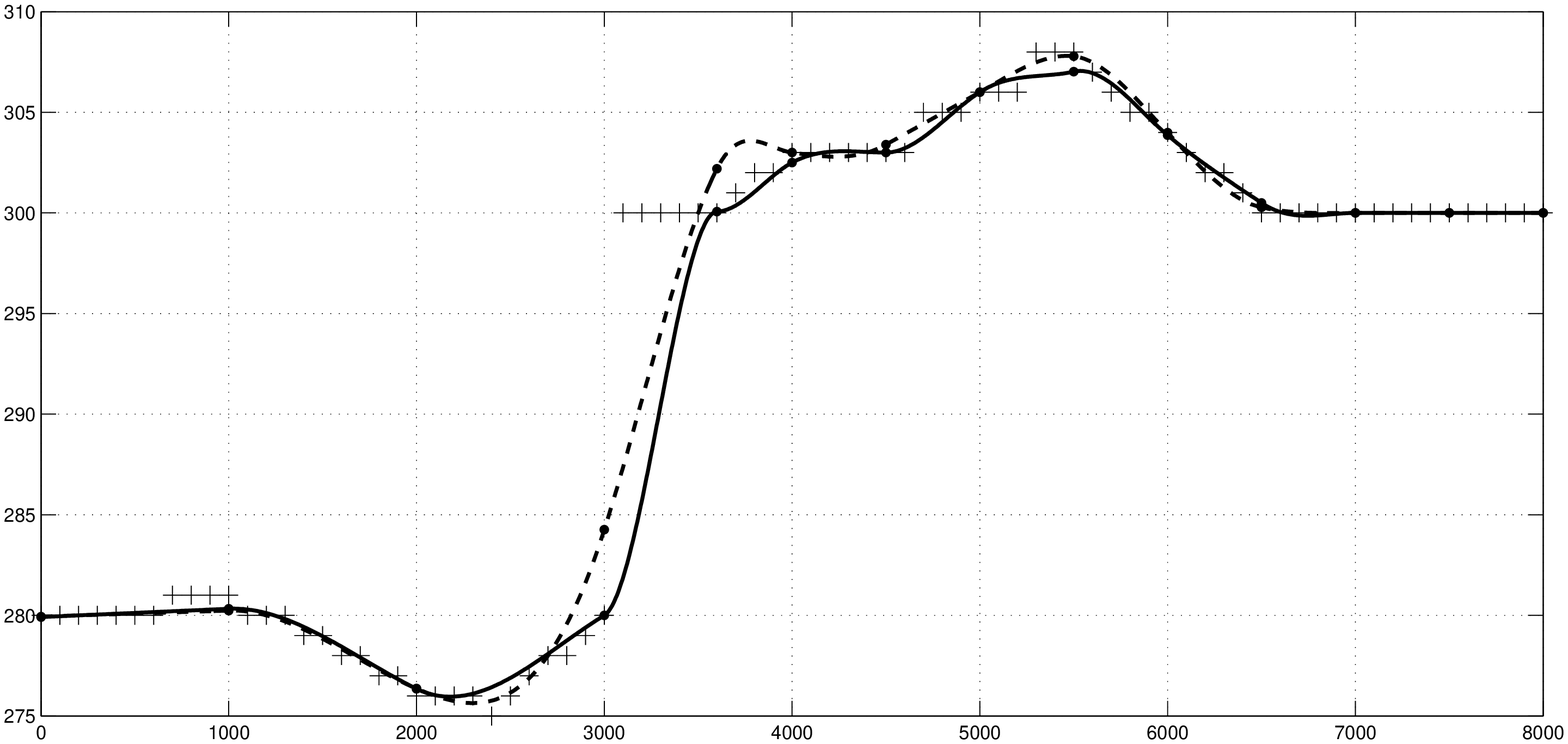}

Discrete $L_1$SFL3\\

\includegraphics[width=0.9\linewidth]{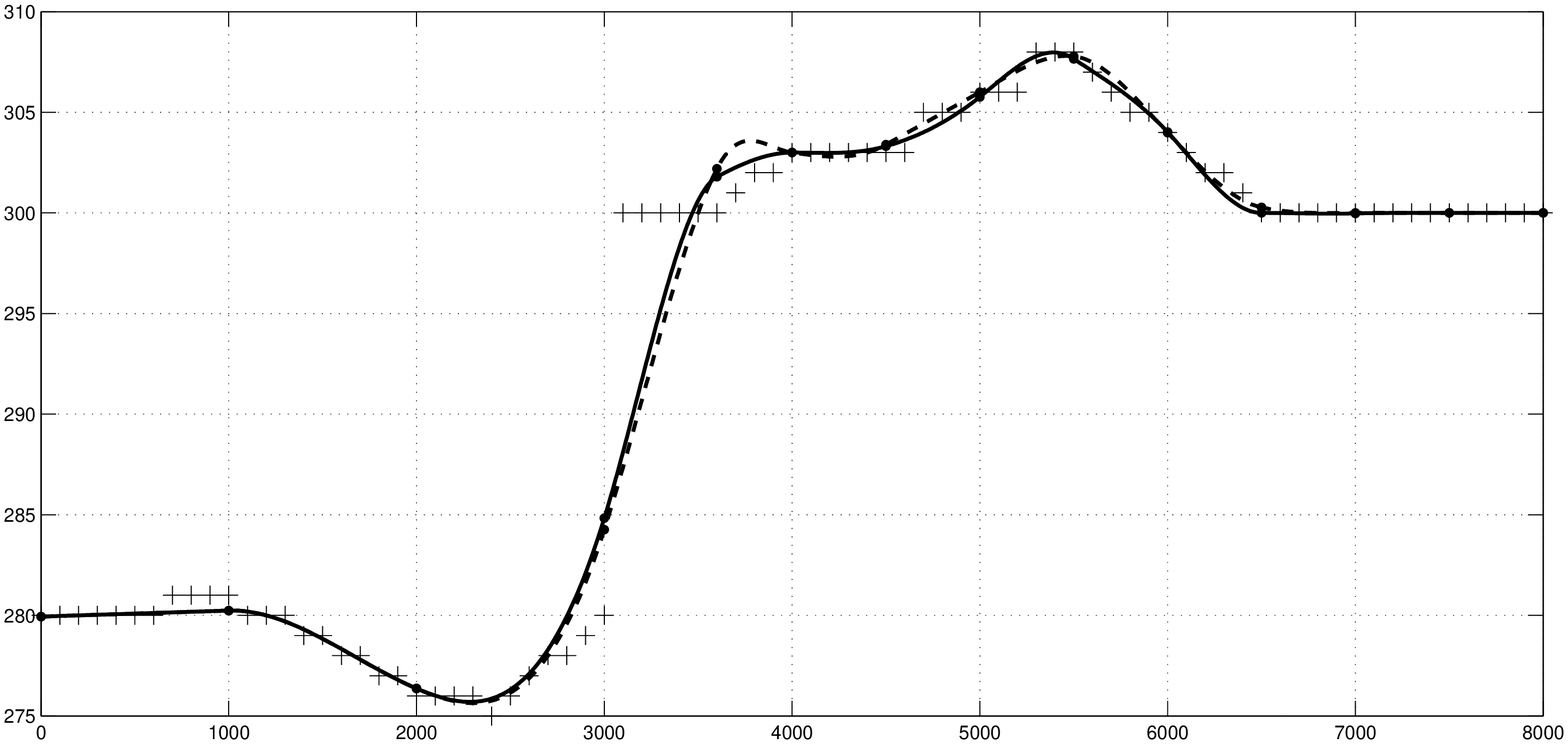}

Discrete $L_1$SFL5\\

\includegraphics[width=0.9\linewidth]{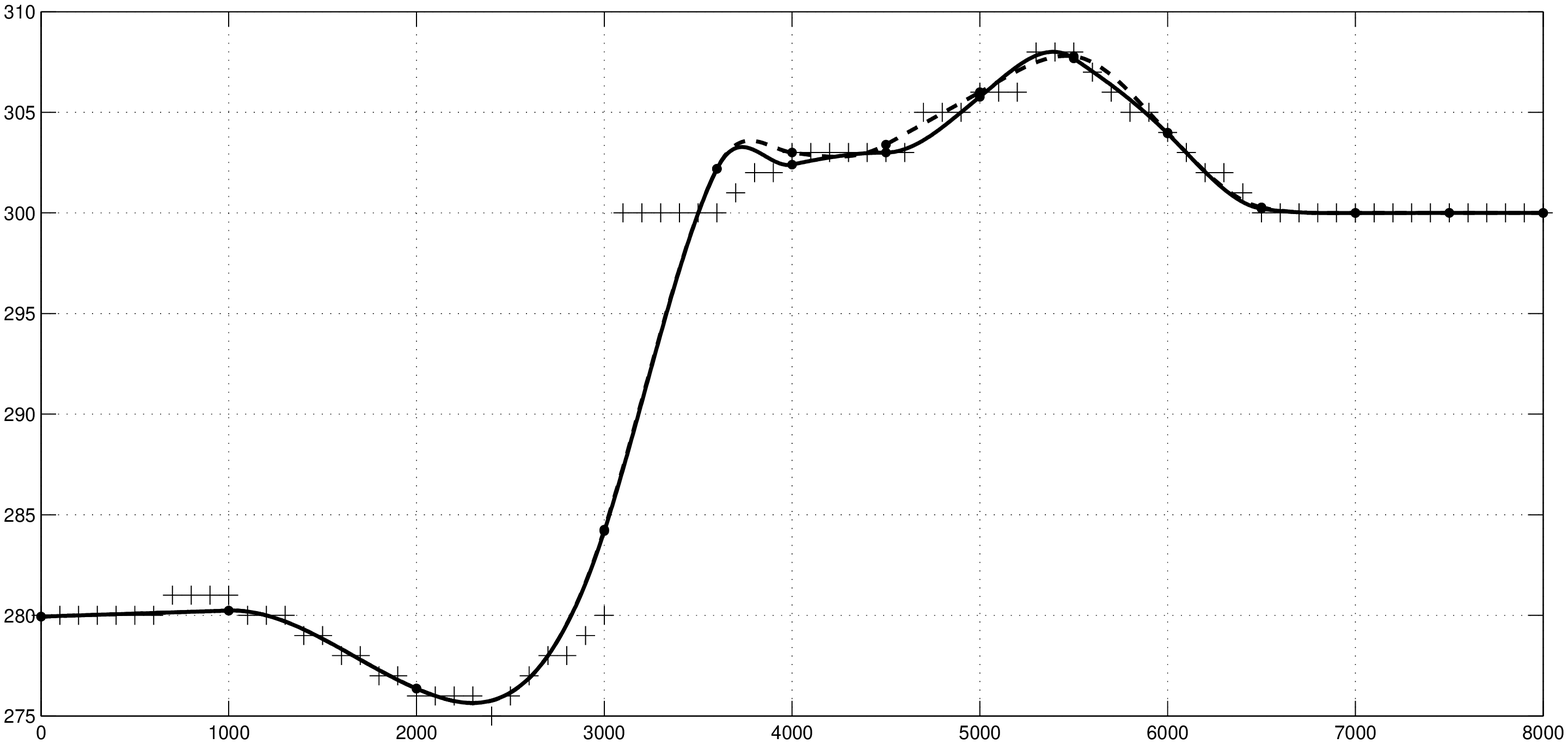}

Discrete $L_1$SFL7
\caption{Local (solid lines) and global (dashed line) $L_1$ spline fits on a multi-scale data set.}
\label{fig_dataset3_L1SFL}
\end{figure}

Like in the continuous case, the seven-point method is the closest graphically to the global method. In some cases like Fig. \ref{fig_dataset1_L1SFL}, linear shape are preserved in a better way.

\section{Modification of $L_1$SFL5 and $L_1$SFL7}

The methods presented before may exhibit on multiscale configurations some undesirable features. We have observed them with the discrete $L_1$SFL5 in Fig.\ref{fig_dataset1_L1SFL} and \ref{fig_dataset2_L1SFL} and with the discrete $L_1$SFL3 in Fig.\ref{fig_dataset3_L1SFL}. This is typically due to a lack of consistency between the different windows. To reduce this phenomenon, we propose two others sliding window methods, $L_1$SFL5-3 and $L_1$SFL7-3, which are respectively a five-point and a seven-point method. The difference with the previous $L_1$SFL5 and $L_1$SFL7 is that we keep now the three middle pieces of information (approximation points and derivative value) instead of one single information. In other words, for sets of $m$ consecutive knots $\mathbf{x}_{i,m}=\{x_{i-\lfloor \frac{m}{2}\rfloor},\dots,x_i,\dots,x_{i+\lfloor \frac{m}{2}\rfloor}\}$ with $i$ going from $\lfloor \frac{m}{2}\rfloor+1$ to $n-\lfloor \frac{m}{2}\rfloor$ by step 3, we will determine numerically a cubic Hermite spline $s_{i,m}^*$ solution of :
\begin{equation}
  \min_{\gamma\in \mathcal{F}_{\mathbf{x}_{i,m}}}  \sum_{j=i-\lfloor \frac{m}{2}\rfloor}^{i+\lfloor \frac{m}{2}\rfloor} \vert s(\hat{x}_j) - \hat{y}_j \vert.
\end{equation}
 Then we keep information at the three central knots:
 \begin{itemize}
   \item $z_{i-1}=s_{i-1,m}^*(x_{i-1})$, $z_{i}=s_{i,m}^*(x_i)$ and $z_{i+1}=s_{i+1,m}^*(x_{i+1})$.
   \item $b_{i-1}=s_{i-1,m}^{*'}(x_{i-1})$,  $b_{i}=s_{i,m}^{*'}(x_i)$ and  $b_{i+1}=s_{i+1,m}^{*'}(x_{i+1})$.
 \end{itemize}

\begin{figure}[!h]
\centering
\includegraphics[width=\linewidth]{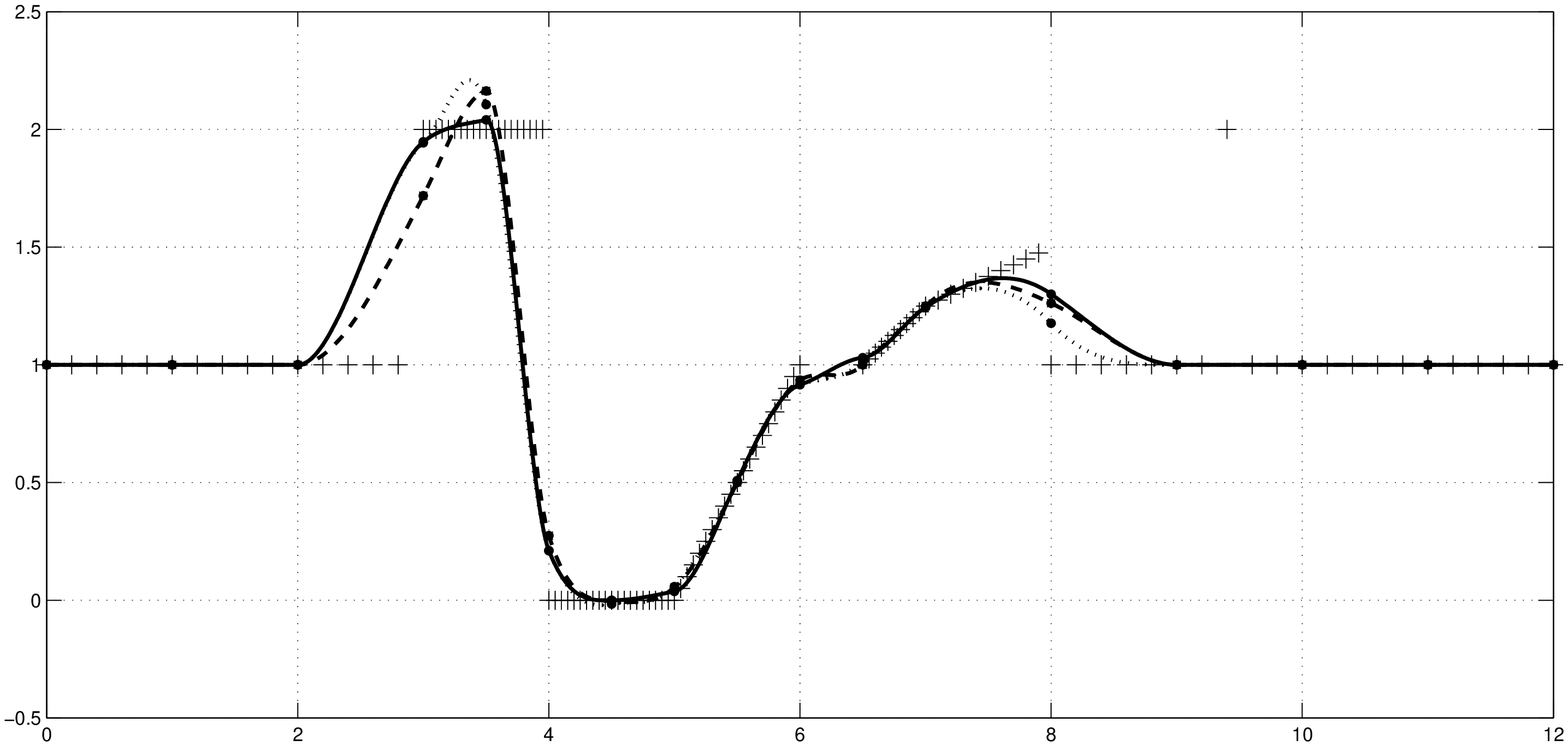}

Discrete $L_1$SFL5-3\\

\includegraphics[width=\linewidth]{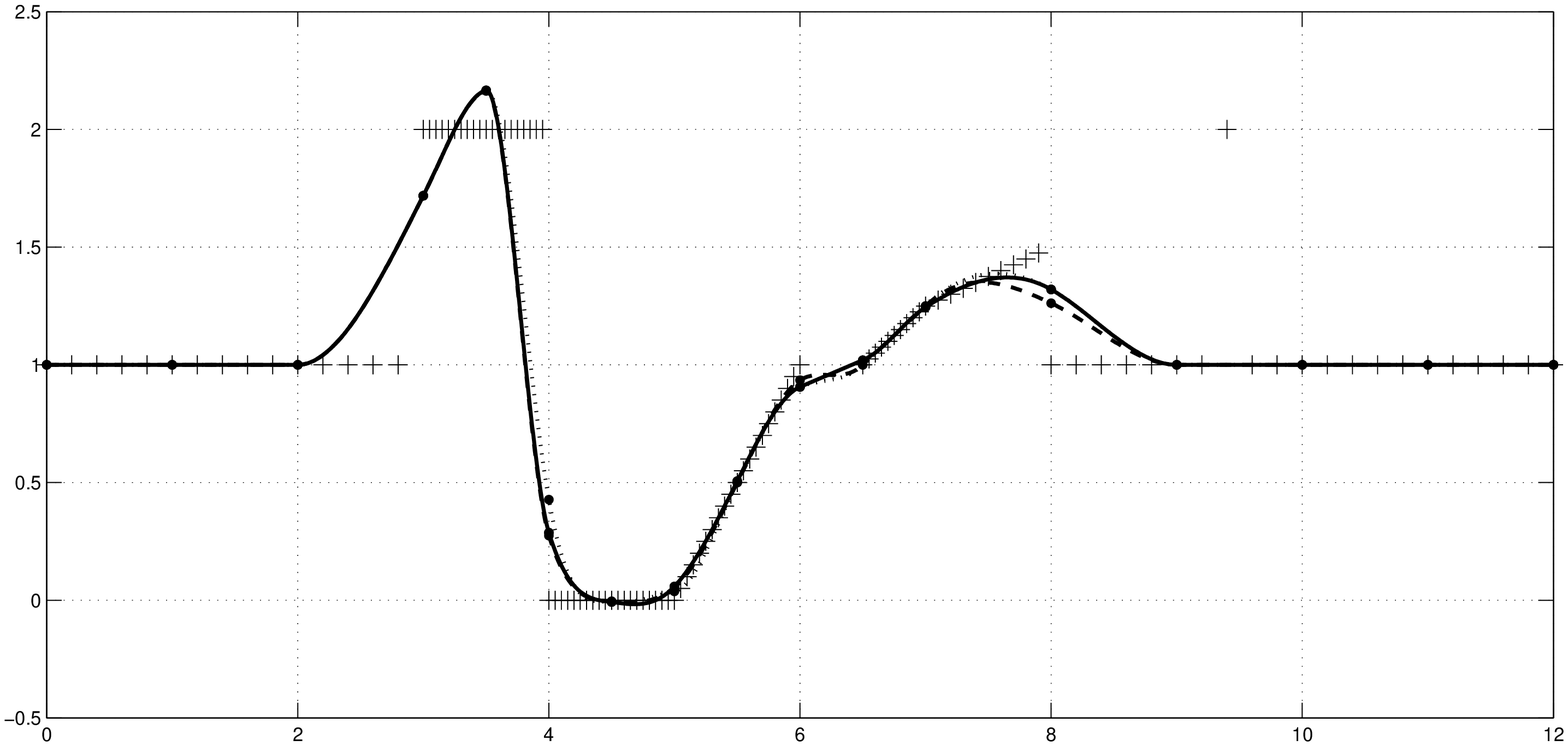}

Discrete $L_1$SFL7-3\\
\caption{Application of discrete $L_1$SFL5-3 and $L_1$SFL7-3 (solid line) on a multiscale dataset. Comparison with previous discrete $L_1$SFL5 and $L_1$SFL7 (dotted line) and global $L_1$SF (dashed line).}
\label{fig_dataset2_L1SFL-3}
\end{figure}

These methods have also the advantage of requiring less computation than the previous $L_1$SFL5 and $L_1$SFL7. Indeed, with $L_1$SFL5-3 and $L_1$SFL7-3, the window slides more rapidly since we do not treat as before every sequence of five, resp. seven, consecutive knots.\\

By this way, we were able to enhance consistency in the five point solution. However, the seven-point method is still the closest one to the initial global method. Since the global method is for now our reference, we select this method for further tests on noisy datasets.\\

We have applied firstly in Fig.\ref{fig_test_heav} our $L_1$SFL7-3 method to a 100-point configuration initially evenly distributed on the Heaviside function and then corrupted by a Gaussian noise with zero mean and 0.03 standard deviation.\\

\begin{figure}[!h]
\centering
\includegraphics[width=\linewidth]{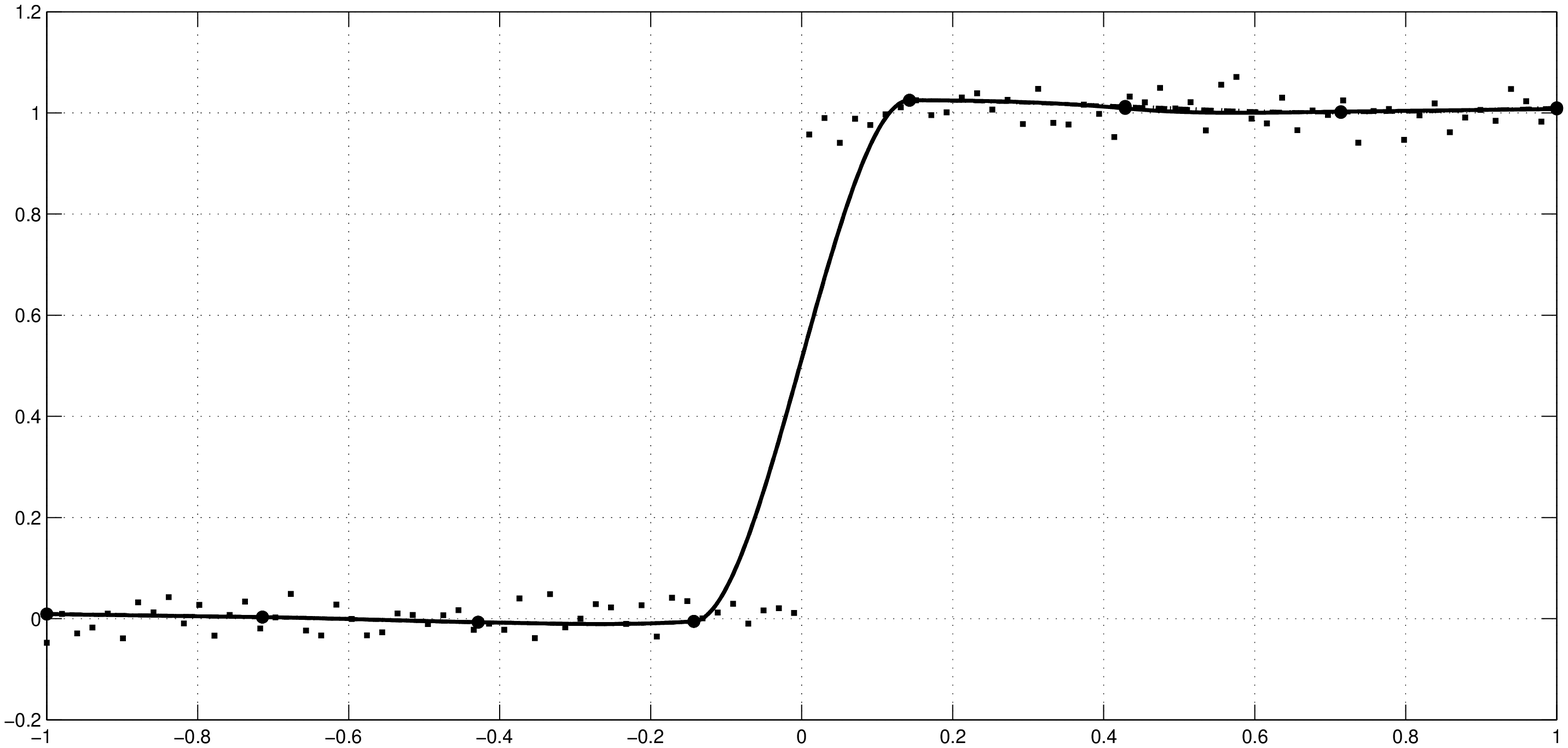}

Discrete $L_1$SFL7-3\\

\includegraphics[width=\linewidth]{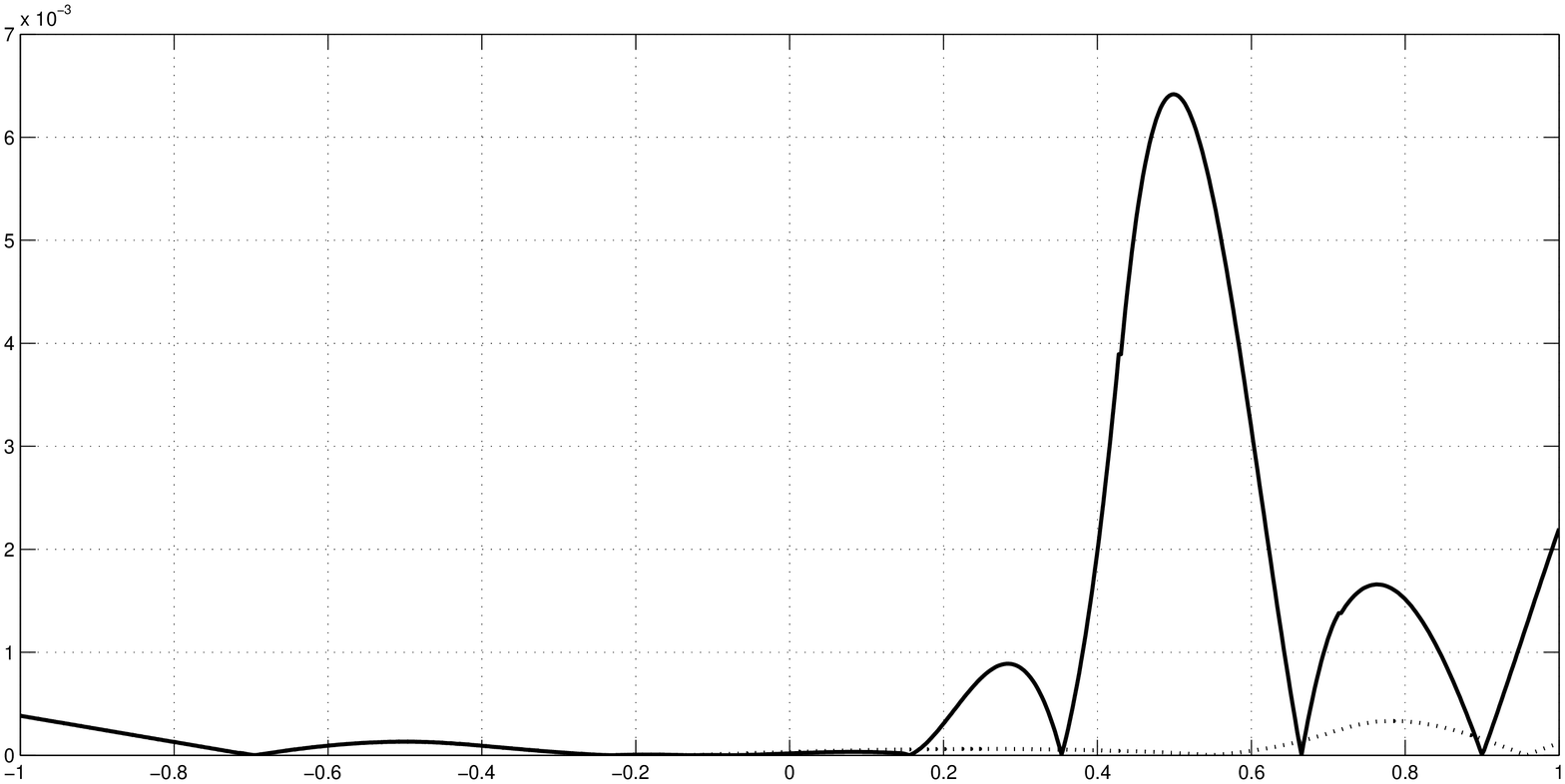}

Pointwise  error to the global solution
\caption{Application of discrete $L_1$SFL7-3 (solid line) on a noisy Heaviside-like dataset. Comparison with previous discrete $L_1$SFL7 (dotted line) and global $L_1$SF (dashed line).}
\label{fig_test_heav}
\end{figure}

Results are compared with the global method and the $L_1$SFL7 method. Solutions are not identical but are similar as the error plot in Fig.\ref{fig_test_heav} suggests it. We have then applied the method on a 300-point configuration lying initially in the sine function and then corrupted by a Gaussian noise with zero mean and 0.05 standard mean. The observations are the same and graphical results are given in Fig.\ref{fig_test_sine}.

\begin{figure}[!h]
\centering
\includegraphics[width=\linewidth]{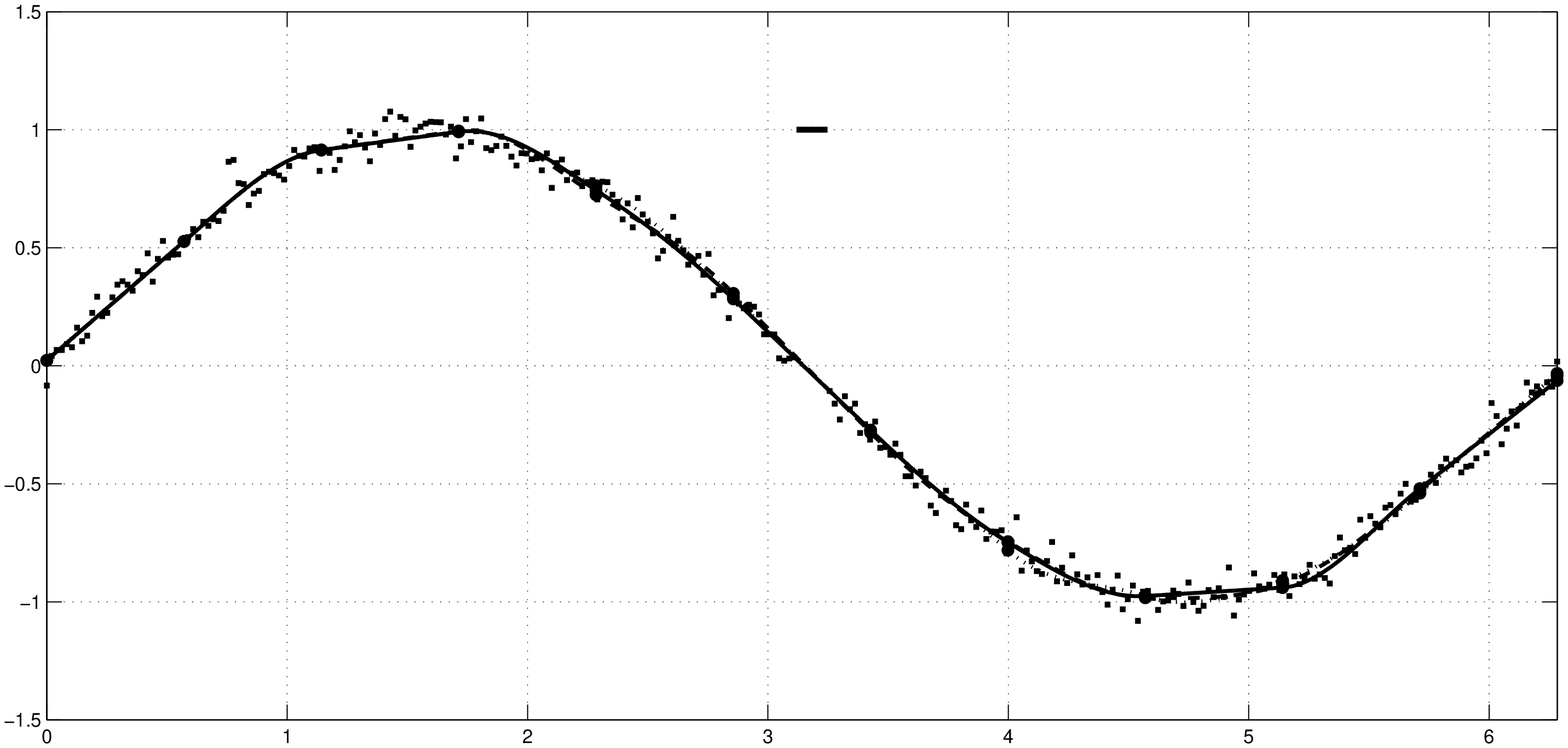}

Discrete $L_1$SFL7-3\\

\includegraphics[width=\linewidth]{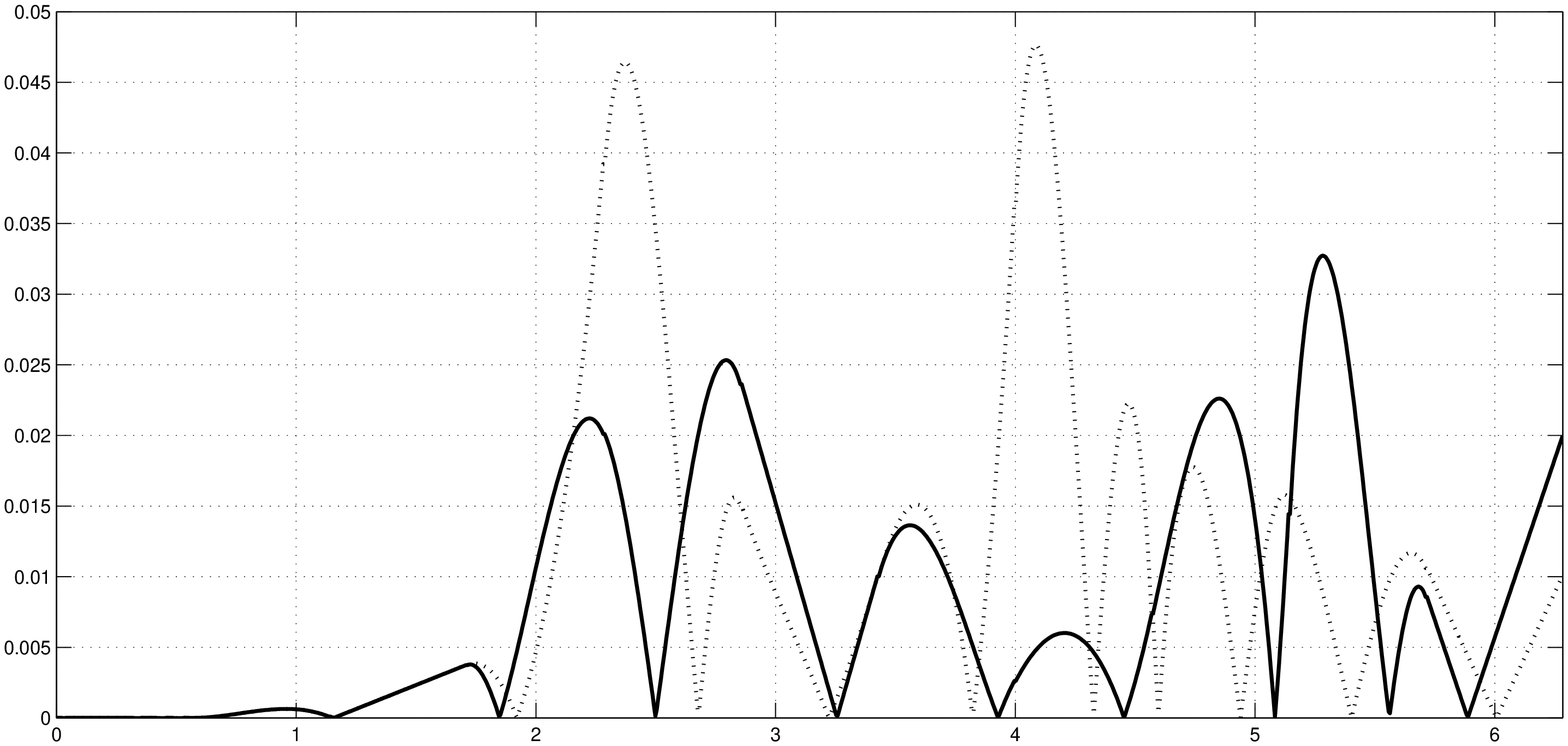}

Pointwise  error to the global solution
\caption{Application of discrete $L_1$SFL7-3 (solid line) on a noisy sinus-like dataset. Comparison with previous discrete $L_1$SFL7 (dotted line) and global $L_1$SF (dashed line).}
\label{fig_test_sine}
\end{figure}

\newpage
\section{Conclusion}
In this article, we have shown the existence of $L_1$ splines fits which are very efficient to approximate data with abrupt changes but time-consuming. In order to obtain lower algorithmic complexity methods, we have tested different methods of computation of $L_1$ spline fits by sliding window process for both continuous and discrete case. At the end of this study, a seven-point method named $L_1$SFL7-3 should be chosen. It is currently a good compromise between keeping the geometrical properties of global $L_1$ spline fits and decreasing computations. The method has linear computational complexity and can be parallelized. This method has shown good results on both multiscale datasets and noisy datasets.

\section{Acknowledgments}
The authors thank deeply Shu-Cherng Fang and Ziteng Wang from the Industrial and Systems Engineering Department of North Carolina State University and John E. Lavery, retired from the Army Research Office, for their comments and suggestions that improved the contents of this paper.

\newpage
\bibliographystyle{alpha}
\bibliography{BibArXiV}
\end{document}